\documentclass[12pt, draft, twoside]{article}

\usepackage[english]{babel}
\usepackage{amsthm}
\usepackage{amsmath}
\usepackage{amssymb}
\usepackage{amsfonts}

\theoremstyle{plain}
\newtheorem{theorem}{Theorem}
\newtheorem{lemma}{Lemma}

\theoremstyle{definition}
\newtheorem{definition}{Definition}
\newtheorem{remark}{Remark}

\numberwithin{equation}{section} 
\numberwithin{theorem}{section}
\numberwithin{lemma}{section} 
\numberwithin{remark}{section}

\newcommand{\Adj}{ \operatorname {Adj} }
\newcommand{\tr}{ \operatorname {tr}  }
\newcommand{\loc}{ {\rm loc} }
\newcommand{\R}{ \mathbb{R} }
\newcommand{\M}{ \mathbb{M} }
\newcommand{\A}{ \mathcal{A} }
\newcommand{\N}{ \mathcal{N} }

\def\?{%\advance\questioncount by 1%
?\vadjust{\vbox to 0pt{\vss\hbox{\kern\hsize\kern1em\large\bf ?!}}}}

\begin{document}

\title{
Variational problems of nonlinear elasticity % theory
in certain classes of mappings with finite distortion
\thanks{
	This work was partially supported by 
	Grant of the Russian Foundation of Basic Research 
	(Project~№~14--01--00552) 
	and  
	Grant of the Russian Federation for the State Support of Researches 
	(Agreement~№~14.B25.31.0029).
}}

\date{\empty}

\author{A.~O.~Molchanova, S.~K.~Vodop$'$yanov}

\newpage

\maketitle

% == Содержание ==
\tableofcontents

\begin{abstract}
	We study the problem of minimizing the functional 
	$$
		I(\varphi)=\int\limits_{\Omega} W(x,D\varphi)\,dx
	$$
	on a~new class of mappings.
	We relax summability conditions for admissible deformations to
	$\varphi\in W^1_n(\Omega)$
	and growth conditions on the integrand
	$W(x,F)$.
	To compensate for that,
	we impose the finite distortion condition
	and the~condition
	$\frac{|D\varphi(x)|^n}{J(x,\varphi)}  \leq M(x) \in L_{s}(\Omega)$,
	$s>n-1$,
	on the characteristic of distortion.
	On assuming that
	the~integrand
	$W(x,F)$
	is polyconvex and coercive, 
	we obtain an~existence theorem
	for the problem of minimizing the functional
	$I(\varphi)$
	on a~new family
	of admissible deformations.\\
	{\bf Keywords:} functional minimization problem,
	nonlinear elasticity,
	mapping with finite distortion,
	polyconvexity.
\end{abstract}

\section{Introduction}\label{sec:intro}
Some problems in nonlinear elasticity
(for instance, 
for hyperelastic materials)
reduce to %the problem of
minimizing the total energy functional. 
In this situation,
in contrast to the case of linear elasticity,
the integrand is almost always nonconvex,
while the functional is nonquadratic.
This renders the standard variational methods inapplicable.
Nevertheless,
for a~sufficiently large class of applied nonlinear problems,
we may replace convexity with certain weaker conditions.
In 1952 Charles Morrey suggested to consider
{\it quasiconvex} 
functions
(see~\cite{Mor1952} for more details). 
Denote by
$\M^{m\times n}$
the set of
$m\times n$
matrices.
A~bounded measurable function
$f: \M^{m\times n}\to \R \cup \{\infty\}$
is called {\it quasiconvex}~if
\begin{equation}\label{def:quasiconvex}
	\int\limits_{\Omega}f(F_0+D\zeta(x))\,dx \geq
	\int\limits_{\Omega}f(F_0)\,dx=|\Omega| \cdot f(F_0).
\end{equation}
holds for each constant $m\times n$ matrix $F_0$
for each bounded open set 
$\Omega \subset \R^n$
and for all  
$\zeta\in C_0^{\infty}(\Omega, \R^m)$.

Morrey showed that
if
$W(\cdot, \cdot)$
is quasiconvex 
and satisfies certain smoothness and growth conditions
then the problem of minimizing the total energy functional
\begin{equation}\label{def:energy_func}
	I(\varphi)=\int\limits_{\Omega}W(x,D\varphi) + \Theta(x, \varphi) \,dx 
\end{equation}
% see~\cite[p.~337]{Ball1977}
has a~solution,
where
$W(x,D\varphi)$
is the stored-energy function
and  
$\Theta(x, \varphi)$
is a~body force potential.
Conversely,
if there exists a~mapping minimizing the~functional
in the class
$C^1(\overline{\Omega})$
functions satisfying given boundary conditions,
then the quasiconvexity condition
\eqref{def:quasiconvex}
is fulfilled.
Although Morrey's results are significant for the theory,
the conditions imposed turn out too restrictive,
excluding applications to
an important class of nonlinear elasticity problems.

In 1977 John Ball developed
another successful approach to nonlinear elasticity problems
using the concept of {\it polyconvexity}
(see~\cite{Ball1977} for more details).
Even though every polyconvex function is also quasiconvex,
weaker growth conditions than in Morrey's articles
are used to prove the existence theorem.

Ball's method is to consider a~sequence
$\{\varphi_k\}$
minimizing the total energy functional~\eqref{def:energy_func}
over the set of admissible deformations
\begin{multline}\label{def:AB}
	\A_B=\{\varphi\in W^1_1(\Omega), I(\varphi) < \infty, \:
	\varphi|_\Gamma=\overline{\varphi}|_\Gamma
	\text{ a.~e.~in } \Gamma=\partial\Omega, \\
	J(x,\varphi) > 0 
	\text{ a. e.~in } \Omega\},
\end{multline}
where
$\overline{\varphi}$
are Dirichlet boundary conditions, 
on assuming that
the {\it coercivity inequality}
\begin{equation}\label{neq:coer_b}
	W(x,F)\geq \alpha (\|F\|^p+\|\Adj F\|^{q}+ (\det F)^r) + g(x)
\end{equation}
holds for almost all
$x\in \Omega$
and all
$F\in \M^n_+$,
where
$p>n-1$,
$q\geq \frac{p}{p-1}$,
$r>1$
and 
$g\in L_1(\Omega)$,
while
$\M^n_+$
stands for the set of matrices of size~%
$n \times n$
with positive determinant.
Moreover,
the stored-energy function~%
$W$
is {\it polyconvex},
that is,
there exists a~convex function
${G(x,\cdot):\M^n\times\M^n \times\R_{+} \to \R}$
such that
\begin{equation*}
	G(x, F, \Adj F, \det F)=W(x,F)
	\text{ for all } F\in \M^n_+
\end{equation*}
almost everywhere in~%
$\Omega$.
By coercivity,
the sequence
$(\varphi_k,\Adj D\varphi_k, J(\varphi_k))$
is bounded in the reflexive Banach space
$W^1_p(\Omega)\times L_{q}(\Omega)\times L_r(\Omega)$.
Hence,
there exists a~subsequence weakly converging to an~element
$(\varphi_0,\Adj D\varphi_0, J(\varphi_0))$.
For the limit
$\varphi_0$
to lie in the class
$\A_B$
of admissible deformations,
we need to impose the additional condition: 
\begin{equation} \label{cond:barrier}
	W(x,F)\rightarrow\infty
	\text{ as }
	\det F \rightarrow 0_+
\end{equation}
(see~\cite{BallCurOl1981} for more details). %~\cite[p.~377]{Ciarlet1992}
This condition is quite reasonable 
since it fits in with the principle that
{\it ``infinite stress must accompany extreme strains''.}
Another important property in this approach
is the {\it sequentially weak lower semicontinuity}
of the total energy functional,
\begin{equation*}
	I(\varphi)\leq \varliminf\limits_{k\rightarrow\infty} I(\varphi_k),
\end{equation*}
which holds
because the stored-energy function is polyconvex.
It is also worth noting that
Ball's approach admits the {\it nonuniqueness of solutions}
observed experimentally
(see~\cite{Ball1977} for more details).

Philippe Ciarlet and Indrich Ne\v{c}as studied~\cite{CiarNec1987} 
injective deformations,
imposing the {\it injectivity condition}
\begin{equation*}
	\int\limits_{\Omega} J(x,\varphi)\, dx \leq |\varphi(\Omega)|
\end{equation*}
on the admissible deformations
and requiring extra regularity
($p>n$).
Under these assumptions,
there exists an~almost everywhere injective minimizer
of the total energy functional.

In the case of problems with boundary conditions on displacement
the injectivity condition turns out superfluous
when the deformation on the boundary
coincides with a~homeomorphism
and the stored-energy function~%
$W$
tends to
infty
sufficiently fast.
Ball obtained this result in 1981~\cite{Ball1981}
(a~misprint is corrected in Exercise~7.13 of~\cite{Ciarlet1992}).
More exactly,
take a~domain
$\Omega \subset \R^3$
and a~polyconvex stored-energy function
$W:\Omega\times\M^3_+ \to \R$.
Suppose that
there exist constants
$\alpha>0$,
$p>3$,
$q > 3$,
$r>1$,
and
$m> \frac{2q}{q-3}$,
as well as a~function
$g\in L_1(\Omega)$
such that
\begin{equation}\label{neq:qoerB}
	W(x,F)\geq \alpha (\|F\|^p+\|\Adj F\|^{q}+ (\det F)^r + (\det F)^{-m}) + g(x)
\end{equation}
for almost all
$x\in \Omega$
and all
$F\in \M^n_+$.
Take a~homeomorphism
$\overline{\varphi}: \overline{\Omega} \to \overline{\Omega'}$
in
$W^1_p(\Omega)$.
Then there exists a~mapping
$\varphi: \Omega \to \Omega'$
minimizing the total energy functional~\eqref{def:energy_func} 
over the set of admissible deformations~\eqref{def:AB},
which is a~homeomorphism.

We should note that
the above conditions on the adjoint matrix
$\Adj Df \in L_q (\Omega)$,
with
$q>3$,
and
$\frac{1}{\det Df} \in L_m (\Omega)$,
where
$m> \frac{2q}{q-3}$,
in fact constrain the~inverse mapping
(since
$Df^{-1} = \frac {\Adj Df} {\det Df}$). 
In this article we discard constraints on the inverse mapping
and consider a~new class of admissible deformations:  
\begin{multline}\label{def:A}
	\A=\{\varphi \in W^1_1(\Omega) \cap FD(\Omega),\: I(\varphi)<\infty, \:
	\frac{|D\varphi(x)|^n}{J(x,\varphi)} < M(x) \in L_{s} (\Omega),\: \\
	s>n-1, \: \varphi|_\Gamma=\overline{\varphi}|_\Gamma
	\text{ a.~e.~in } \Gamma,
	\: J(x,\varphi) \geq 0 
	\text{ a.~e.~in } \Omega \},
\end{multline}
where
$FD(\Omega)$
is the class of mappings with finite distortion. 
Another significant difference
is the weakening of conditions on the stored-energy function.
The coercivity inequality becomes
\begin{equation} \label{neq:coer1}
	W(x,F)\geq \alpha (\|F\|^n+ (\det F)^r) + g(x).
\end{equation}

In some previous works
the deformation~%
$\varphi$ 
was required 
to lie in the~Sobolev class 
$W^1_p(\Omega)$
with
$p>n$
when there is a~compact embedding of~%
$W^1_p(\Omega)$
into the space of continuous functions
$C(\Omega)$,
that is,
assume at the~outset that~%
$\varphi$
is continuous.
In this article we only assume that
$\varphi \in W^1_n(\Omega)$;
consequently,
we must prove separately that
the admissible deformation is continuous. 
The second feature 
of this article
is the replacement of conditions on the ``inverse mapping''
by the summability of the distortion coefficient 
$\frac{|D\psi(x)|^n}{J(x,\psi)} < M(x) \in L_{s} (\Omega)$,
where
$s>n-1$.

The main result of this article is the following theorem.

\begin{theorem}
	Given a~polyconvex function
	$W(x,F)$
	satisfying the coercivity inequality~\eqref{neq:coer1},
	a~homeomorphism
	$\overline{\varphi}: \overline{\Omega} \to \overline{\Omega'}$,
	$\overline{\varphi} \in W^1_n(\Omega)$,
	and a~nonempty set~%
	$\A$,
	there exists at least one mapping
	$\varphi_0\in\A$
	such that
	$I(\varphi_0)=\inf\limits_{\varphi\in\A} I(\varphi)$.
	Moreover,
	$\varphi_0: \overline{\Omega} \to \overline{\Omega'}$
	is a~homeomorphism.
\end{theorem}

It is not difficult to verify that
for sufficiently large values of the exponents
$q>n(n-1)s$
and
$m > \frac{sq(n-1)^2}{q-sn(n-1)}$
the summability of the distortion coefficient
$\frac{|D\psi(x)|^n}{J(x,\psi)} \in L_{s} (\Omega)$
of the mapping
$\psi \in \A$
follows from %the inequality
\eqref{neq:qoerB} and Corollary~6 of~\cite{Vod2012}.
We should note also that
in this article the property of mapping to be sense  preserving
follows from the property that
the required deformation is a~mapping with bounded
$(n,q)$-distortion
\cite{BayVod2015}.

Naturally,
the proofs of our main results %of this article
differ substantially from Ball's methods in~\cite{Ball1977, Ball1981}
and depend crucially on the results and methods of~\cite{Vod2012}.

In the first section %of the article
we present the concept of polyconvexity of functions.
The second section contains auxiliary facts.
The third section is devoted to the main result:
the existence theorem and its proof.
In the fourth section we give two examples.
In the first example  
we consider a~stored-energy function
$W(F)$
for which both minimization problems
(in the classes
$\A_B$
and~%
$\A$)
have solutions.
The second example discusses a~function
$W(F)$
violating the coercivity 
\eqref{neq:coer_b} 
and asymptotic condition 
\eqref{cond:barrier};
nevertheless,
there exists a~solution to the minimization problem in~% the class~%
$\A$.

The results of this article 
were announced in the note~\cite{VodMol}
together with a~sketch of the proof of the main result.

\section{The concept of polyconvexity}\label{sec:concept_poly}
For a~large class of physical problems 
it may be assumed that
the stored-energy function is {\it polyconvex}.
\begin{definition}[\cite{Ball1977}]%\cite[Def 4.2]{Ball1977},~\cite[p.174]{Ciarlet1992}
	A~function
	$W:\mathbb{F}\to \R$
	defined on an~arbitrary subset
	$\mathbb{F}\subset \M^3$
	is {\it polyconvex}
	if there exists a~convex function
	$G:\mathbb{U}\to \R$,
	where
	\begin{equation*}
		\mathbb{U}=\{(F, \Adj{F}, \det F)\in \M^3\times \M^3
		\times \R,\: F\in \mathbb{F} \},
	\end{equation*}
	such that
	\begin{equation*}
		G(F, \Adj F, \det F)=W(F)
		\text{ for all } F\in \mathbb{F}.
	\end{equation*}
	Here
	$\Adj F$
	stands for the adjugate matrix,
	that is,
	the transpose of cofactor matrix.		 
\end{definition}

As examples of polyconvex
but not convex functions,
consider
\begin{equation*}
	W(F)= \det F
\end{equation*}
and
\begin{equation*}
	W(F)=\tr \Adj F^T F = \|\Adj F^T F\|^2.
\end{equation*}

\subsection{Polyconvex stored-energy functions}\label{sec:poly}
\label{sec:2}

Consider the stored-energy function %of the form
\begin{equation*}
	W(F)=\sum\limits_{i=1}\limits^{M}a_i {\rm tr} \: (F^T F)^{\frac{\gamma_i}{2}}
	+
	\sum\limits_{j=1}\limits^{N}b_j \Adj (F^T F)^{\frac{\delta_j}{2}} + \Gamma(\det F)
\end{equation*}
with
$a_i>0$,
$b_j>0$,
$\gamma_i\geq 1$,
$\delta_j \geq 1$,
and a~convex function
$\Gamma:(0,\infty)\to \R$.
If
\begin{equation*}
	\lim\limits_{\delta \rightarrow 0_+} \Gamma(\delta)=\infty
\end{equation*}
then the hyperelastic material is called an~{\it Ogden material}
\cite{Ogden1972}.
These materials have polyconvex stored-energy function
and satisfy the growth conditions of~%
Ball's existence theorem.

Ogden materials are interesting not only in theory,
but also in practice.
Moreover
(see~\cite{Ciarlet1992} for more details),
for a~hyperelastic material with experimentally known Lam\'e coefficients
it can be constructed a~stored-energy function of an~Ogden material.

\subsection{Non-polyconvex stored-energy functions}\label{sec:nonpoly}

Saint-Venant--Kirhhoff materials
are well-known examples of hyperelastic materials.
Its stored-energy function is %of the form
\begin{equation*}
	W(F)=\frac{\lambda}{2}(\tr E)^2 + \mu \tr E^2,
\end{equation*}
where
$\lambda$~%
and~%
$\mu$
are Lam\'e coefficients
and
\begin{equation*}
	I+2E=F^T F.
\end{equation*}
This function is a~particular case of the function %of the form
\begin{equation*}%\label{nepoly}
	W(F)=a_1\,\tr F^T F + a_2\,\tr (F^T F)^2 + b \,\tr \Adj F^T F + c
\end{equation*}
with
$a_1 < 0$,
$a_2 > 0$,
and
$b > 0$.
Although this function %(\ref{nepoly})
resembles the~function of an Ogden material
and satisfies the coercivity inequality,
it is not polyconvex 
\cite[Theorem 4.10]{Ciarlet1992}.

\section{Preliminaries}\label{sec:preliminary}

In this section we present some important concepts and statements 
necessary to proceed.
On a~domain
$\Omega\subset\R^n$
we define in a~standard way
(see~\cite{Maz2011} for instance) 
the spaces
$C_0^{\infty}(\Omega)$
of smooth 
functions with compact support,
the~Lebesgue spaces
$L_p(\Omega)$
and
$L_{p, \loc}(\Omega)$
of integrable functions,
and the~Sobolev spaces
$W^1_p(\Omega)$
and
$W^1_{p, \loc}(\Omega)$.

For  working with a~geometry of domains we need the following 
definition.

\begin{definition}\label{quasiisom}
	A homeomorphism  
	$\varphi:\Omega\to \Omega'$ 
	of two open sets 
	$\Omega$, 
	$\Omega'\subset\R^n$ 
	is called a{\it~quasi-isometric mapping} 
	if the following inequalities 
	\begin{equation*}%\label{qi1}
		\varlimsup\limits_{y\to x}\frac{|\varphi(y)-\varphi(x)|}{|y-x|}\leq M \quad \text{ and } \quad 
		\varlimsup\limits_{y\to z}\frac{|\varphi^{-1}(y)-\varphi^{-1}(z)|}{|y-z|}\leq M
	\end{equation*}
	hold for all 
	$x\in \Omega$ 
	and 
	$z\in \Omega'$ 
	where  
	$M$ 
	is some constant independent of the~choice of points 
	$x\in \Omega$ 
	and 
	$z\in \Omega'$.
\end{definition}

Recall the following definition.

\begin{definition}
	A~domain
	$\Omega\subset\R^n$
	is called a~domain with {\it locally quasi-isometric  boundary}
	whenever for every point
	$x\in\partial \Omega$
	there are a~neighborhood
	$U_x \subset\R^n$
	and a~quasi-isometric mapping
	$\nu_x: U_x \to B(0, r_x) \subset \R^n$,
	where the~number
	$r_x>0$
	depends on~%the neighborhood
	$U_x$,
	such that
	$\nu (U_x \cap \partial \Omega) \subset \{y \in B(0, r_x) \mid y_n = 0\}$.
\end{definition}

\begin{remark}\label{equivbound}
	In some papers 
	it is used 
	a bi-Lipschitz mapping
	$\varphi$ 
	instead of  
	quasi-isometric mapping in this definition.
	It is evident that 
	the bi-Lipschitz mapping	 
	is also  quasi-isometric one.
	The inverse implication is not valid but it is valid the following assertion:%
	\textit{ every quasi-isometric
	mapping is locally bi-Lipschitz one} 
	(see a proof below). 
	Hence  $\Omega$  is a  domain 
	with locally Lipschitz boundary
	if and only if it is a domain 
	with  quasi-isometric boundary.
\end{remark}

\begin{proof}
	For proving this statement fix a quasi-isometric mapping 
	$\varphi: \Omega \to \Omega'$. 
	We have to verify for any fixed  ball 
	$B\Subset \Omega$ 
	the inequality 
	$$
	d_{\varphi(B)}(\varphi(x),\varphi(y))\leq L|\varphi(x)-\varphi(y)|
	$$ 
	holds for all points 
	$x$,
	$y\in B$ 
	with some constant 
	$L$ 
	depending on the choice of 
	$B$ 
	only (here 
	$d_{\varphi(B)}(u,v)$ 
	is the intrinsic metric in the domain 
	$\varphi(B)$ 
	defined as the infimum over the lengths of all rectifiable curves in 
	$\varphi(B)$ 
	with endpoints 
	$u$ 
	and 
	$v$)\footnote{It is well-known that a mapping is quasi-isometric iff the lengths of a rectifiable curve in the domain and of its image 
	are comparable. The last property is equivalent to the following one:  given mapping $\varphi: \Omega \to \Omega'$ is quasi-isometric iff  $ L^{-1}  d_B(x,y)\leq d_{\varphi(B)}(\varphi(x),\varphi(y))\leq L d_B(x,y)$ for all $x,y\in B$.}. 
	Take an arbitrary function 
	$g\in W^1_\infty(\varphi(B))$. 
	Then 
	$\varphi^*(g) = g\circ \varphi\in W^1_\infty(B)$ 
	and by Whitney type extension theorem 
	(see for instance \cite{Vod1988,Vod1989}) 
	there is a bounded  extension operator 
	$\operatorname{ext}_B: W^1_\infty(B)\to W^1_\infty(\R^n)$. 
	Multiply 
	$\operatorname{ext}_B(\varphi^*(g))$ 
	by a cut-of-function 
	$\eta\in C_0^\infty(\Omega)$ 
	such that
	$\eta(x) =1$ 
	for all points 
	$x\in B$. 
	Then the product 
	$\eta\cdot\operatorname{ext}_B(\varphi^*(g))$ 
	belongs to 
	$W^1_\infty(\Omega)$, 
	equals~0 near the boundary 
	$\partial \Omega$ 
	and its norm in 
	$W^1_\infty(\Omega)$ 
	is controlled by the norm 
	$\|g\mid W^1_\infty(\varphi(B))\|$. 
	  
	It is clear that 
	$\varphi^{-1}{}^*(\eta\cdot\operatorname{ext}_B(\varphi^*(g)))$ 
	belongs to 
	$W^1_\infty(\Omega')$, 
	equals~0 near the boundary 
	$\partial \Omega'$ 
	and its norm in 
	$W^1_\infty(\Omega')$ 
	is controlled by the norm 
	$\|g\mid W^1_\infty(\varphi(B))\|$.  
	Extending 
	$\varphi^{-1}{}^*(\eta\cdot\operatorname{ext}_B(\varphi^*(g)))$ 
	by~0 outside~%
	$\Omega'$ 
	we obtain a~bounded extension operator
	$$
		\operatorname{ext}_{\varphi(B)}: W^1_\infty(\varphi(B))\to 
		W^1_\infty(\R^n).
	$$
	It is well-known (see for example \cite{Vod1988,Vod1989}) 
	that a necessary and sufficient condition for existence of such 
	operator is an~equivalence of the interior metric in  
	$\varphi(B)$ 
	to the Euclidean one:
	the inequality  
	$$
		d_{\varphi(B)}(u,v)\leq L |u-v|
	$$
	holds for all points 
	$u$,
	$v\in \varphi(B)$ 
	with some constant 
	$L$.
\end{proof}

Taking into account Remark~\ref{equivbound}
we can consider also a domain with  
quasi-isometric boundary instead of domain with Lipschitz boundary
in  the statements  formulated below.

\begin{theorem}[Rellich--Kondrachov theorem,
see~\cite{Adams1975} for instance]  
\label{thm:embedding}
	Consider a~bounded domain~%
	$\Omega$
	in
	$\R^n$
	and
	$1 \leq p<\infty$.
	If~%
	$\Omega$
	satisfies the cone condition
	then the following embeddings are compact: 
	\begin{enumerate}
	\item
		$W^1_p(\Omega)\Subset L_q(\Omega)$
		for
		$1\leq q<p^*=\frac{np}{n-p}$
		with
		$p<n$;

	\item
		$W^1_p(\Omega)\Subset L_q(\Omega)$
		for
		$1\leq q<\infty$
		with
		$p=n$.

	\item
		If~$\Omega$
		has a~locally Lipschitz boundary
		$\partial \Omega$
		then for
		$p>n$
		the embedding
		$W^1_p(\Omega)\Subset C(\overline{\Omega})$
		is compact.
	\end{enumerate}
\end{theorem}

\begin{theorem}[Properties of the trace operator~\cite{Maz2011}]
\label{thm:trace} 
	Consider a~bounded domain~%
	$\Omega$
	with locally Lipschitz boundary
	$\partial \Omega$
	endowed with the
	$(n-1)$-dimensional
	Hausdorff measure
	$\mathcal{H}^{n-1}$
	and
	$1\leq p < \infty$.
	There exists a~bounded linear operator
	$\tr$
	such that
	$\tr f = f$
	on
	$\partial \Omega$
	for all
	$f \in W^1_p(\Omega) \cap C(\overline{\Omega})$, with properties:

	\begin{enumerate}
	\item if
		$1\leq p<n$
		then
		$\tr:W^l_p(\Omega) \to L_q(\partial \Omega)$
		for
		$1<q<p^*=\frac{(n-1)p}{n-p}$
		and furthermore,
		for
		$1< p<n$
		the operator
		$\tr$
		is compact;

	\item if
		$p=n$
		then
		$\tr:W^l_p(\Omega) \to L_q(\partial \Omega)$
		for
		$1<q<\infty$
		and furthermore,
		the operator
		$\tr$
		is compact;	

	\item if
		$n<p$
		then
		$\tr:W^l_p(\Omega) \to C(\partial \Omega)$.
		and furthermore,
		the operator
		$\tr$
		is compact.
	\end{enumerate}
\end{theorem}

We also need the following theorem.
\begin{theorem}[Corollary to Poincar\'e inequality,
	see~\cite{Ciarlet1992} for instance] 
	\label{thm:Poin} 
	Given a~connected bounded domain
	$\Omega \subset\R^n$
	with Lipschitz boundary 
	$\Gamma=\partial \Omega$,
	a~measurable subset
	$\Gamma_0$
	of~%
	$\Gamma$
	with
	$|\Gamma_0|> 0$,
	and
	$1\leq p <\infty$,
	there exists a~constant
	$C_1$
	such that
	\begin{equation*}
		\int\limits_{\Omega} |f(x)|^p\, dx \leq C_1 \bigg(
		\int\limits_{\Omega} |Df(x)|^p\, dx +
		\bigg|\int\limits_{\Gamma_0} f(x)\, ds \bigg|^p \bigg).
	\end{equation*}
	for all
	$f\in W^1_p(\Omega)$.
\end{theorem}

\begin{lemma}[\cite{Resh1982, BojIwa1983}]\label{lem:free_div}
	Take an~open connected set
	$\Omega\subset \R^n$
	with
	$n\geq 2$
	and a~mapping
	$f: \Omega \to \R^n$
	with
	$f\in W^1_p(\Omega)$,
	for
	$p>n-1$.
	Then the columns of the matrix
	$\Adj Df (x)$
	are divergence-free vector fields,
	that is,
	$\sum\limits_{j=1}\limits^{n}
 	\frac{\partial}{\partial x_j} A_{jk} = 0$
	in the sense of distributions 
	for all
	$k=1,\dots, n$.
\end{lemma}

\begin{remark}
	The proofs in~\cite{Resh1982, BojIwa1983}
	rest on smooth approximations to a~Sobolev mapping.
	The new proof of this lemma in~\cite{Vod2007}
	avoids these approximations.
\end{remark}

We also need the following corollary to the Bezikovich theorem
(see Theorem~1.1 in~\cite{Gus1978} for instance).

\begin{lemma} \label{lem:bez}
	For every open set
	$U \subset \R^n$
	with
	$U\neq \R^n$
	there exists a~countable family
	$\mathcal{B}=\{B_j\}$
	of balls such that
	\begin{enumerate}
	\item 
		$\bigcup\limits_{j} B_j = U$;

	\item
		if
		$B_j=B_j(x_j,r_j) \in \mathcal{B}$
		then
		${\rm dist} (x_j, \partial U) = 12 r_j$;

	\item
		the families
		$\mathcal{B} = \{B_j\}$
		and
		$2 \mathcal{B} = \{2B_j\}$,
		where the symbol
		$2B$
		stands for the ball of doubled radius centered at the same point,
		constitute a~finite covering of~% the set 
		$U$;

	\item
		if the balls
		$2 B_j = B_j(x_j, 2 r_j)$, 
		$j=1$,
		$2$
		intersect
		then
		$\frac{5}{7} r_1 \leq r_2 \leq \frac{7}{5} r_1$;

	\item
		we can subdivide the family
		$\{2B_j\}$
		into finitely many tuples
		so that
		in each tuple the balls are disjoint
		and the number of tuples depends only on the dimension~%
		$n$.
	\end{enumerate}
\end{lemma}

The main concepts of functional analysis
like weak convergence,
semicontinuous functionals,
and related theorems
are described in detail in~\cite{EkTem1976, KolmFom1976}.
Let us recall some of them.

\begin{theorem} \label{thm:conv}
	For a~normed vector space~%
	$V$
	and a~Banach space~%
	$W$
	consider a~continuous bilinear mapping
	$B: V\times W \to \R$.
	If
	$v_k \rightarrow v$
	strongly and 
	$w_k \rightarrow w$
	weakly
	then
	$B(v_k, w_k) \rightarrow B(v,w)$.
\end{theorem}

\begin{theorem}[Mazur theorem,
	see~\cite{EkTem1976} for instance]  
	\label{thm:Mazur}
	Let 
	$v_k \rightarrow v$
	weakly in a~normed vector space
	$V$.
	Then there exist convex combinations
	\begin{equation*}
		w_k=\sum\limits_{m=k}\limits^{N(k)} \lambda_m^k v_k, \quad
		\text{where} \quad
		\sum\limits_{m=k}\limits^{N(k)} \lambda_m^k =1, 
		\quad \lambda_m^k\geq 0, 
		\quad k\leq m\leq N(k),
	\end{equation*}
	converging to~%
	$v$
	in norm.
\end{theorem}

\begin{definition}
	A~function
	$J: V\to \R \cup \{\infty\}$
	is called {\it sequentially weakly lower semicontinuous}
	whenever
	\begin{equation*}
		J(u_0)\leq\varliminf\limits_{u_k\rightarrow u_0} J(u_k)
	\end{equation*}
	for every weakly converging sequence
	$\{u_k\}\subset V$.
\end{definition}

\begin{definition}
	Let  an~open set
	$\Omega\subset \R^k$,
	a~mapping 
	$G:\Omega \times \R^{m} \to \overline{\R}$
	enjoys the~{\it Carath\'eodory conditions}
	whenever 

	{\bf (a)}
		$G(x, \cdot)$
		is continuous on
		$\R^{m}$
		for almost all
		$x\in \Omega$;

	{\bf (b)}
		$G(\cdot, a)$
		is measurable on~%
		$\Omega$
		for all
		$a\in\R^{m}$.
\end{definition}

The reader not familiar with mappings with bounded distortion
may look at~\cite{Rickman1993,Resh1982}.
Let us recall the main concepts and theorems.

\begin{theorem}[\cite{Resh1982}]\label{thm:Resh}
	Consider a~sequence
	$f_m: \Omega \to \R^n$
	of mappings of class
	$W^1_{n, \loc} (\Omega)$.
	Suppose that
	the following conditions are met:

	{\rm (a)}
	$\{f_m\}$
	is locally bounded in
	$W^1_{n, \loc} (\Omega)$;

	{\rm (b)}
	$\{f_m\}$
	converges in
	$L_1(\Omega)$
	to some mapping
	$f_0$
	as
	$m\rightarrow \infty$.

	Then %the limit mapping
	$f_0\in W^1_{n,\loc}(\Omega)$
	and
	\begin{equation*}
		\int\limits_{\Omega} \varphi(x) J(x,f_m) \, dx \longrightarrow
		\int\limits_{\Omega} \varphi(x) J(x,f_0) \, dx \quad 
		\text{ as } m\longrightarrow \infty
	\end{equation*}
	for every continuous real function
	$\varphi:U\to \R$
	with compact support in~%
	$\Omega$.
\end{theorem}

\begin{definition}[\cite{IwaSve1993}]
	A~mapping
	$f:\Omega\to \R^n$
	with
	$f \in W^1_{1, \loc}(\Omega)$
	is called a~{\it mapping with finite distortion},
	written
	$f \in FD (\Omega)$,
	whenever
	$J(x,f) \geq 0$
	almost everywhere in 
	$\Omega$
	and
	\begin{equation*}
		|Df(x)|^n \leq K(x) J(x,f) \quad
		\text{for almost all}\: x\in \Omega,
	\end{equation*}
	where
	$0< K(x) < \infty$
	almost everywhere in 
	$\Omega$.
\end{definition}

\begin{remark}
	In other words,
	the {\it finite distortion} condition 
	amounts to the~vanishing of the partial derivatives of %the mapping
	$f \in W^1_{1, \loc}(\Omega)$
	almost everywhere on the zero set of the Jacobian.
\end{remark}

\begin{definition}[\cite{Resh1982}]
	Given
	$f:\Omega\to \R^n$
	with
	$f\in W^1_{n,\loc}(\Omega)$
	and
	$J(x,f)\geq 0$ 
	almost everywhere in 
	$\Omega$,
	the {\it distortion} of~%
	$f$
	at~%
	$x$
	is the function
	\begin{equation*}
		K(x)=
		\begin{cases}
			\frac{|Df(x)|^n}{J(x,f)} & \text{if } J(x,f)>0,
			\\
			1 & \text{if } J(x,f)=0.
		\end{cases}
	\end{equation*}
\end{definition}

\begin{theorem}[\cite{ManVill1998}]\label{thm:ManVill}
	Let 
	$f\in W^1_{n,\loc}(\Omega)$
	be nonconstant mapping
	whose dilatation
	$K(x)$
	is in
	$L_{s, \loc}(\Omega)$.
	Then, if 
	$s>n-1$
	and
	$J(x,f)\geq 0$ for almost all $x \in \Omega$,
	the mapping 
	$f$
	is continuous,
	discrete,
	and open. 
\end{theorem}

\begin{remark}
	Theorem 2.3 of~\cite{VodGold1976} shows that
	this mapping is continuous. 
\end{remark}

\begin{definition}[\cite{Vod2005}]
	{\it a~mapping 
	$\varphi:\Omega\to \Omega'$ 
	induces a~bounded operator
	$\varphi^*:L^1_p(\Omega^{\prime})\to L^1_q(\Omega)$ 
	by the  composition rule}, 
	$1\leq q\leq p<\infty$, 
	if the~following properties fulfil:

	(a) If two quasicontinuous functions 
		$f_1$, $f_2\in L_{p}^{1}(\Omega^{\prime})$ 
		are distinguished on a~set of 
		$p$-capacity 
		zero, 
		then the functions
		$f_1\circ\varphi$, 
		$f_2\circ\varphi$ 
		are distinguished on a~set of
		measure zero;

	(b) If  
		$\tilde f\in L_{p}^{1}(\Omega^{\prime})$ 
		is a~quasicontinuous representative of 
		$f$,
		then
		$\tilde f\circ\varphi\in L_{p}^{1}(\Omega)$ 
		but 
		$\tilde f\circ\varphi$ 
		is not required
		to be quasicontinuous;

	(c) The mapping  
	$\varphi^*: f\mapsto\tilde f\circ\varphi$, 
	where
	$\tilde f$ 
	is a~quasicontinuous representative of 
	$f$,
	is a~bounded operator 
	$L^1_p(\Omega^{\prime})\to L^1_q(\Omega)$.
\end{definition}

The definition and properties 
$p$-capacity 
see, for instance, in
\cite{CV1996,Maz2011,Resh1969}.

\begin{theorem}[{\cite[Theorem 4]{VodUhl2002}}]\label{thm4}
	Take two open sets
	$\Omega$~%
	and~%
	$\Omega'$
	in
	$\R^n$
	with
	$n\geq 1$.
	If a~mapping
	$\varphi: \Omega \to \Omega'$
	induces  a~bounded composition operator
	\begin{equation*}
		\varphi^*: L^1_p(\Omega') \cap C^\infty(\Omega')\to L^1_q (\Omega), 
		\quad 1\leq q \leq p \leq n,
	\end{equation*}
	with
	$\varphi(f)=f\circ \varphi$
	then~%
	$\varphi$
	has Luzin $\N^{-1}$-property.
\end{theorem}

\begin{remark}
	Theorem~4 of~\cite{VodUhl2002} is stated for a~mapping
	$\varphi: \Omega \to \Omega'$
	generating a~bounded composition operator
	$\varphi^*: L^1_p(\Omega') \to L^1_q (\Omega)$
	with
	$1\leq q \leq p \leq n$.
	Observe that
	only smooth test functions are used in its proof,
	%Thus,
	%the proof of Theorem~4 of~\cite{VodUhl2002}
	which therefore
	also justifies Theorem~\ref{thm4}.
\end{remark}

Following~\cite{Vod2012},
for a~mapping
$f:\Omega\to\Omega'$
of class
$W^1_{1,\loc}(\Omega)$
define the{\it~distortion operator function}
\begin{equation*}
	K_{f,p} (x) =
	\begin{cases}
		\frac{|Df(x)|}{|J(x,f)|^{\frac{1}{p}}} &
		\text{ for } 
		x\in \Omega\setminus(Z\cap\Sigma),\\
		0 &
		\text{ otherwise},
	\end{cases}
\end{equation*}
where~%
$Z$
is the zero set of the Jacobian
$J(x,f)$
and
$\Sigma$
is a~singularity set,
meaning that
$|\Sigma|=0$
and~%
$f$
enjoys Luzin $\N$-property
outside~%
$\Sigma$.

\begin{theorem}[\cite{Vod2012, VodUhl1998, VodUhl2002}] \label{thm:Vod5}
	A~homeomorphism
	$\varphi:\Omega\to\Omega'$
	induces a~bounded composition operator 
	\begin{equation*}
	\varphi^*: L^1_p(\Omega')\to L^1_q(\Omega),
	\quad
	1\leq q \leq p <\infty,
	\end{equation*}
	where
	$\varphi^*(f)=f\circ\varphi$
	for
	$f\in L^1_p(\Omega')$,
	if and only if
	the following conditions are met:
	\begin{enumerate}
	\item
		$\varphi \in W^1_{q, \loc}(\Omega)$;
	\item
		the mapping~%
		$\varphi$
		has finite distortion;
	\item
		$K_{\varphi,p}(\cdot) \in L_{\varkappa}(\Omega)$,
		where
		$\frac{1}{\varkappa}=\frac{1}{q}-\frac{1}{p}$, $1\leq q \leq p <\infty$
		{\rm(}and
		$\varkappa=\infty$
		for
		$q=p${\rm)}.
	\end{enumerate}

	Moreover,
	\begin{equation*}
		\|\varphi^*\|\leq\|K_{\varphi,p}(\cdot) \mid L_{\varkappa}(\Omega)\|
		\leq C \|\varphi^*\|
	\end{equation*}
	for some constant~%
	$C$.
\end{theorem}

\begin{remark}
	Necessity is proved in~\cite{VodUhl1998, VodUhl2002} 
	(see also earlier work~\cite{Ukhlov1993}), 
	and sufficiency,
	in Theorem~5 of~\cite{Vod2012}.
\end{remark}

\begin{theorem}[\cite{Vod2012}] %(\cite[Theorem 6]{Vod2012})
\label{thm:Vod6}
	Assume that
	a~homeomorphism
	$\varphi:\Omega\to\Omega'$
	\begin{enumerate}
	\item
		induces a~bounded composition operator
		$\varphi^*:L^1_p(\Omega') \to L^1_q(\Omega)$
		for
		$n-1\leq q \leq p \leq \infty$,
		where
		$\varphi^*(f)=f\circ\varphi$
		for
		$f\in L^1_p(\Omega')$,
	\item
		has finite distortion for
		$n-1\leq q \leq p = \infty$.
	\end{enumerate}

	Then the inverse mapping
	$\varphi^{-1}$
	induces a~bounded composition operator
	$\varphi^{-1*}:L^1_{q'}(\Omega) \to L^1_{p'}(\Omega')$,
	where
	$q'=\frac{q}{q-n+1}$
	and
	$p'=\frac{p}{p-n+1}$,
	and has finite distortion.

	Moreover,
	\begin{equation*}
		\|\varphi^{-1*}\|\leq\|K_{\varphi^{-1},q'}(\cdot) \mid L_{\rho}(\Omega')\|
		\leq
		\|K_{\varphi,p}(\cdot) \mid L_{\varkappa}(\Omega)\|^{n-1},
		\text{ where } 
		\frac{1}{\rho}=\frac{1}{p'}-\frac{1}{q'}.
	\end{equation*}
\end{theorem}

\begin{definition} 
	Consider
	$f: \Omega \to \R^n$
	and
	$D \subset \Omega$.
	The function 
	$N_f(\cdot, D): \R^n \to \mathbb{N} \cup \{\infty\}$
	defined as 
	\begin{equation*}
		N_f(y, D) =
		\textrm{card}(f^{-1}(y) \cap D)
	\end{equation*}
	is called the \textit{Banach indicatrix}.
\end{definition} 

\begin{theorem}[Change-of-variable formula{\cite[Theorem 2]{Haj1993}}]\label{thm:change_var}
	Given an~open set
	$\Omega \subset \R^n$,
	if a~mapping
	$f: \Omega \to \R^n$
	is approximatively differentiable almost everywhere on~%
	$\Omega$ 
	then we can redefine~%
	$f$
	on a~negligible set to gain Luzin 
	$\N$-property.

	If a~mapping
	$f: \Omega \to \R^n$
	is approximatively differentiable almost everywhere on~%
	$\Omega$
	and has Luzin $\N$-property
	then for every measurable function
	$u: \R^n \to \R$
	and every measurable set
	$D \subset \Omega$
	we have:
	\begin{enumerate}
	\item
		the functions
		$(u \circ f) (x) |J(x, f)|$
		and
		$u(y) N_f(y, D)$
		are measurable;
	\item
		if
		$u \geq 0$
		then 
		\begin{equation}\label{change_var}
			\int\limits_{D} (u \circ f) (x) |J(x, f)| \, dx
			= \int\limits_{\R^n} u (y) N_f(y, D) \, dy;
		\end{equation}
	\item
		if one of the functions
		$(u \circ f) (x) |J(x, f)|$
		and
		$u(y) N_f(y, D)$
		is integrable
		then so is the second,
		and the change-of-variable formula~\eqref{change_var} holds.
	\end{enumerate}
\end{theorem}

\begin{definition} 
	Consider a~continuous,  open, and discrete mapping 
	$f: \Omega \to \R^n$
	and a~point
	$x\in\Omega$.
	There exists the domain~%
	$V$
	containing~%
	$x$
	such that
	$\overline{V}\cup  f^{-1}(f(\{x\}))=\{x\}$.
	Refer as the {\it local index} of~%
	$f$
	at~%
	$x$
	to the quantity
	$i(x,f)=\mu(f(x),f,V)$.
\end{definition}

\begin{remark}
	The local index is well-defined and independent of the domain~%
	$V$
	under consideration
	\cite{Rickman1993, Resh1982}.
\end{remark}

\begin{definition} 
	Given a~continuous mapping
	$f: \Omega \to \R^n$,
	say that~%
	$f$
	is {\it sense preserving}
	whenever the  topological degree satisfies
	$\mu(y,f,D) > 0$
	for every domain
	$D \Subset \Omega$
	and every point 
	$y\in f(D) \setminus f(\partial \Omega)$.
\end{definition}

\begin{definition} 
	A~mapping
	$f: \Omega \to \R^n$
	is called a~{\it mapping with bounded
	$(p,q)$-distortion},
	with
	$n-1<q\leq p<\infty$,
	whenever
	\begin{enumerate}
	\item
		the mapping~%
		$f$
		is continuous,
		open,
		and discrete;
	\item
		the mapping~%
		$f$
		is of Sobolev class
		$W^1_{q,\loc}(\Omega)$;
	\item
		its Jacobian satisfies
		$J(x,f)\geq 0$
		for almost all
		$x\in \Omega$;
	\item
		the mapping~%
		$f$
		has finite distortion:
		$Df(x) = 0$
		if and only if
		${J(x,f) = 0}$
		with the possible exclusion of a~negligible set;
	\item
		the local distortion function
		$K_{f,p}(x)$
		belongs to 
		$L_\varkappa(\Omega)$ where
		$\frac{1}{\varkappa}=\frac{1}{q}-\frac{1}{p}$
		($\varkappa = \infty$
		for
		$p=q$).
	\end{enumerate}
\end{definition}

\begin{lemma} [\cite{BayVod2015}]\label{lem:orient}
	A~mapping
	$f: \Omega \to \R^n$
	with bounded 
	$(p,q)$-distortion
	is sense preserving
	in the case
	$q>n-1$.
\end{lemma}

\begin{proof}
Indeed,
if~%
$f$
is a~homeomorphism
then
$J(x,f) = 0$
cannot hold almost everywhere on~%
$\Omega$
because this would imply that
$Df(x) = 0$
almost everywhere,
and so~%
$f$
would not be an~open mapping.
Consequently,
$J(x,f) > 0$
on a~set of positive measure
and there exists a~point
$x \in \Omega$
of differentiability of~%
$f$
at which the differential is nondegenerate,
while the~Jacobian is positive
(see~\cite[Proposition 1]{Vod2000}  for instance).
The properties of degrees of mappings imply that~%
$f$
is sense preserving.

The general case reduces to the previous one
since the image of the set of branch points
is closed in
$f(\Omega)$.
Indeed,
take
$D \Subset \Omega$
avoiding the branch points
and a~point~%
$z$
in a~connected component
$U_1$
of the set 
$f(D)\setminus f(\partial D)$.
Put
$D_1 = f^{-1}(U_1) \subset D$.
Since
$D_1$
is an~open set,
the Jacobian
$J(\cdot,f)$
cannot vanish almost everywhere on it
(otherwise,
%the open set
$D_1$
would map to a~point in contradiction with the openness of~%
$f$).
Hence,
$J(x,f) > 0$
on a~set of positive measure.
For a~point
$x_0\in D_1$
in this set
we have
$y=f(x_0) \in U_1$.
Consider
$f^{-1}(y) \cap D_1 = \{x_0, x_1, \dots x_N\}$.
Since~%the mapping
$f$
is a~discrete mapping,
there exists a~ball
$B(y,r)$
of small radius~%
$r$
with %such that 
$f^{-1}(B(y,r)) = \cup W_j$,
where
$x_j \in W_j$
and
$W_i \cap W_k = \emptyset$.
Since
$f:W_j \to B(y,r)$
is a~homeomorphism,
the~local index
$i(x_j,f)=\mu(y,f,W_j)$
is positive, 
while the degree satisfies
\begin{equation*}
	\mu(x,f,D_1) = \sum_{j=1}^{N} i(x_j,f) = 
	\sum_{j=1}^{N}\mu(y,f,B_j) >0,
\end{equation*}
see Proposition~4.4 of~\cite{Rickman1993}.

Now take
$D \Subset \Omega$
intersecting the set~%
$V$
of branch points
and suppose that
the image of a~branch point
$z$
lies in the connected component
$U_1$
of %the set
$f(D)\setminus f(\partial D)$.
Since the image of the set of branch points is closed,
there must be points in
$U_1$
which are outside of
$f(V)$.
Applying the previous argument to a~point
$z_1\in U_1 \setminus f(V)$,
we infer that
$\mu(x,f,D_1)>0$.
Thus,
the~mapping~%
$f$
is sense preserving.
\end{proof}

\section{The Main Result}\label{sec:main}
\subsection{Existence theorem}\label{subsec:exist}

Consider two bounded domains
$\Omega$, 
$\Omega'\subset \R^n$
with locally Lipschitz boundaries
$\partial\Omega=\Gamma$
and
$\partial\Omega'=\Gamma'$.
Consider the functional 
\begin{equation*}
	I(\varphi)=\int\limits_{\Omega}W(x, D\varphi(x))\,dx,
\end{equation*}
where
$W: \Omega \times \M^n \to \R$
is a~stored-energy function
with the following properties:

{\bf (a)  polyconvexity:}
	there exists a~convex function
	$G(x, \cdot): \M^n\times \M^n \times \R_{\geq 0}
	\to \R$,
	satisfying Carath\'erode conditions,
	such that
	for all
	$F\in \M^n_{\geq 0}$
	the~equality
	\begin{equation*}
		G(x, F, \Adj F, \det F)=W(x,F)
	\end{equation*}
	holds almost everywhere in~%
	$\Omega$;

{\bf (b) coercivity:}
	there exist constants
	$\alpha>0$
	and
	$r>1$
	as well as a~function
	$g\in L_1(\Omega)$
	such that
	\begin{equation}\label{neq:coer}
		W(x,F)\geq \alpha (\|F\|^n+ (\det F)^r) + g(x)
	\end{equation}
	for almost all
	$x\in \Omega$
	and all
	$F\in \M^n_{\geq 0}$.

Given
$\overline{\varphi}: \Omega \to\Omega'$
with
$\overline{\varphi}\in W^1_n(\Omega)$
and a~measurable function
${M: \Omega \to\R}$,
define the class of admissible deformations
\begin{multline*}
	\A=\{\varphi \in W^1_1(\Omega) \cap FD (\Omega),\: I(\varphi) < \infty, \:
	\frac{|D\varphi(x)|^{n}}{J(x,\varphi)} \leq M(x) \in L_{s} (\Omega), \\
	s>n-1, \: \varphi|_\Gamma=\overline{\varphi}|_\Gamma
	\text{ a.~e.~in } \Gamma,
	\: J(x,\varphi) \geq 0 
	\text{ a.~e.~in } \Omega \}.
\end{multline*}

\begin{remark}
	Here we understand
	$\psi|_\Gamma=\overline{\varphi}|_\Gamma$
	in the sense of traces,
	that is,
	$\psi - \overline{\varphi} \in \overset{\circ}{W^1_n}(\Omega)$.
\end{remark}

\begin{theorem}[Existence theorem] \label{thm:main}
	Suppose that:
	\begin{enumerate}
	\item 
		conditions {\bf (a)}~and~{\bf (b)}
		on the function
		$W(x,F)$
		are fulfilled;
	\item 
		$\overline{\varphi}: \overline{\Omega} \to \overline{\Omega'}$
		is a~homeomorphism;
	\item 
		the set~%
		$\A$
		is nonempty.
	\end{enumerate}

	Then there exists at least one mapping
	$\varphi_0\in\A$
	such that
	\begin{equation*}
		I(\varphi_0)=\inf\limits_{\varphi\in\A} I(\varphi).
	\end{equation*}
	Moreover,
	$\varphi_0: \overline{\Omega} \to \overline{\Omega'}$
	is a~homeomorphism.
\end{theorem}

\subsection{Proof of the existence theorem}\label{subsec:proof}

We subdivide the proof of the existence theorem into three steps.
On the first step (see subsection~\ref{subsec:existence})
we establish that
the weak limit of a~minimizing sequence exists.
On the second step we investigate
the main properties of the~mappings of class~%
$\A$
(subsection~\ref{subsec:prop_phi})
and the limit mapping
$\varphi_0$
(subsections~\ref{subsec:J>0}--%
% \ref{subsec:phi_0_Gamma}, \ref{subsec:bounded_comp}, \ref{subsec:inj_phi},
\ref{subsec:prop_K}),
as well as verify that
$\varphi_0$
belongs to the class~%
$\A$
of admissible deformations.
This is a~key step 
since we introduce a~new class of admissible deformations,
and consequently,
the verification of containment in it
differs substantially from previous works.
Our proof uses both classical theorems of functional analysis
and properties of mappings with finite distortion
obtained quite recently~\cite{Vod2012}.
Finally,
on the third step it remains to show that
the mapping found is actually a~solution to the minimization problem,
which requires proving that
the energy functional is lower semicontinuous
(subsection~\ref{subsec:semicont}).

\subsubsection{Existence of a~minimizing mapping.}\label{subsec:existence}

Let us prove the existence of a~minimizing mapping
for the functional 
\begin{equation*}
	\overline{I}(\varphi)=I(\varphi)- \int\limits_\Omega g(x)\, dx.
\end{equation*}

\begin{lemma}\label{lem:AdjJ}
	If
	$\Omega$
	is a~domain in
	$\R^n$
	with locally Lipschitz boundary
	and
	$\varphi_0$, 
	$\varphi_k\in W^1_n(\Omega)$
	then
	$M^{p}_{q} (D\varphi_0)\in L_{\frac{n}{m}}(\Omega)$,
	where
	$p = (p_1, p_2, \dots, p_m)$
	and
	$q =  (q_1, q_2, \dots, q_m)$,
	while
	$M^{p}_{q} (D\varphi_0)$
	is the minor of the matrix
	$D \varphi_0$
	of size~%
	$m$,
	with
	$1\leq m \leq n$,
	consisting of the entries %of the form
	$\dfrac{\partial (\varphi_0)_{p_i}}{\partial x_{q_j}}$
	{\rm (}that is,
	of the rows
	$p_1, p_2, \dots p_m$
	and columns
	$q_1, q_2, \dots q_m${\rm )},
	and if
	$\varphi_k\rightarrow \varphi_0$
	weakly in
	$W^1_n(\Omega)$,
	while
	$M^{p}_{q} (D\varphi_k) \rightarrow H^p_q$
	weakly in
	$L_{\frac{n}{m}}(\Omega)$
	for all~%
	$m$
	with
	$1\leq m \leq n$,
	then
	$H^p_q=M^p_q (D\varphi)$
	for all
	$m$.
\end{lemma}

\begin{proof}
Since
$D\varphi_0\in L_n(\Omega)$,
applying H\"older's inequality and the boundedness of~% the domain
$\Omega$,
we easily verify that 
$M^{p}_{q} (D\varphi_0)\in L_{\frac{n}{m}}(\Omega)$.

Further we prove by  induction on~%
$m$.
In the case
$m=1$
the assertion %of the theorem
follows directly from the definition of
$\varphi_k$, 
$\varphi_0 \in W^1_n(\Omega)$.

For sufficiently smooth functions
($C^n(\Omega)$)
expanding the determinant along the first row yields
\begin{multline*}
	M^p_q (D f) = \sum \limits_{j=1} \limits^{n} (-1)^{j-1} 
	\dfrac{\partial f_{p_1} }{\partial x_{q_j}} M^{\hat{p}_1}_{\hat{q}_j} 
	=
	\sum \limits_{j=1} \limits^{n} (-1)^{j-1} \dfrac{\partial}{\partial x_{q_j}}
	\Bigl( f_{p_1} M^{\hat{p}_1}_{\hat{q}_j} \Bigr)
	\\ + 
	f_{p_1} \sum \limits_{j=1} \limits^{n} (-1)^{j-1} 
	\dfrac{\partial}{\partial x_{q_j}} 
	M^{\hat{p}_1}_{\hat{q}_j}=
	\sum \limits_{j=1} \limits^{n} (-1)^{j-1} \dfrac{\partial}{\partial x_{q_j}}
	\Bigl( f_{p_1} M^{\hat{p}_1}_{\hat{q}_j} \Bigr),
\end{multline*}
where
$\hat{q}_j = (q_1, q_2, \dots q_{j-1}, q_{j+1}, \dots q_m)$.
The second term in the right-hand side
vanishes not only for smooth functions
(which can be verified directly),
but also for the functions of class
$L_{\frac{n}{m}} (\Omega)$
(Lemma~\ref{lem:free_div}).

Given a~test function
$\theta \in \mathcal{D} (\Omega)$,
for
$\xi \in W^1_{n}(\Omega)$
and
$\eta\in L_{\frac{n}{m-1}}(\Omega)$
the~bilinear mapping
\begin{equation*}
	(\xi, \eta) \mapsto \int \limits_{\Omega} \xi M^{\hat{p}_1}_{\hat{q}_j} (D \eta) 
	\frac{\partial \theta}{\partial x_{q_j}} \, dx
\end{equation*}
is continuous by H\"older's inequality.
Indeed, 
\begin{multline*}
	\left| \int \limits_{\Omega} \xi M^{\hat{p}_1}_{\hat{q}_j} (D \eta) 
	\frac{\partial \theta}{\partial x_{q_j}} \, dx \right| \leq \| \xi \|_{n}  
	\| M^{\hat{p}_1}_{\hat{q}_j} (D \eta) \|_{\frac{n}{m}} \left\| 
	\frac{\partial \theta}{\partial x_{q_j}} \right\|_{\frac{n-1}{n-m-1}}  
	\\ \leq 
	C \| \xi \|_{n}  \| M^{\hat{p}_1}_{\hat{q}_j} (D \eta) \|_{\frac{n}{m}}.
\end{multline*}

Since the embedding of
$W^1_n (\Omega)$
into
$L_p(\Omega)$
is compact,
extracting a~subsequence if necessary,
we may assume that
the sequence
$\varphi_k$
converges strongly in
$L_p(\Omega)$,
while the sequence
$M^{\hat{p}_1}_{\hat{q}_j} (D \eta)$
converges weakly in 
$L_{\frac{n}{m-1}}(\Omega)$
by the~inductive assumption.
Consequently,
Theorem~\ref{thm:conv} yields the convergence
\begin{equation*}
	\int \limits_{\Omega} (\varphi_k)_{p_1} M^{\hat{p}_1}_{\hat{q}_j} (D \varphi_k) 
	\frac{\partial \theta}{\partial x_{q_j}} \,dx 
	\xrightarrow[k \rightarrow \infty]{}  \int \limits_{\Omega} 
	(\varphi_0)_{p_1} M^{\hat{p}_1}_{\hat{q}_j} (D \varphi_0) 
	\frac{\partial \theta}{\partial x_{q_j}}\,dx.
\end{equation*}
Therefore,
\begin{align*}
	\int \limits_{\Omega} M^{p}_{q} (D \varphi_k) \theta \,dx
	& =
	\sum \limits_{j=1} \limits^{n} (-1)^{j-1} \int \limits_{\Omega} 
	\dfrac{\partial}{\partial x_{q_j}} \Bigl( (\varphi_k)_{p_1} 
	M^{\hat{p}_1}_{\hat{q}_j} (D \varphi_k) \Bigr) \theta \,dx
	\\ & = 
	\sum \limits_{j=1} \limits^{n} (-1)^{j}\int \limits_{\Omega} 
	(\varphi_k)_{p_1} M^{\hat{p}_1}_{\hat{q}_j} (D \varphi_k)
	\frac{\partial \theta}{\partial x_{q_j}} \,dx
	\\ & 
	\xrightarrow[k \rightarrow \infty]{} 
	\sum \limits_{j=1} \limits^{n} (-1)^{j}\int \limits_{\Omega} 
	(\varphi_0)_{p_1} M^{\hat{p}_1}_{\hat{q}_j} (D \varphi_0) 
	\frac{\partial \theta}{\partial x_{q_j}}\,dx 
	\\ & = 
	\int \limits_{\Omega} M^{p}_{q} (D \varphi_0) \theta \,dx,
\end{align*}
which completes the proof of the lemma.
\end{proof}

Observe now that
the coercivity~\eqref{neq:coer}
 of the function~%
$W$,
Lemma~\ref{lem:AdjJ},
and Poincar\'e inequality
(Theorem~\ref{thm:Poin})
ensure the existence of constants
$c>0$
and
$d\in\R$
such that
%the inequality
\begin{multline}\label{est:lower}
	\overline{I}(\varphi) = I(\varphi)- \int\limits_\Omega g(x)\, dx  \\
	\geq
	c\biggl(\|\varphi\mid W^1_n(\Omega)\|^n+\|\Adj D\varphi \mid
	L_{\frac{n}{n-1}}(\Omega)\|^{\frac{n}{n-1}} +\|J(\cdot,\varphi) 
	\mid L_r(\Omega)\|^r\biggr)\\
	+ d
\end{multline}
%holds
for every mapping
$\varphi \in \A$.

Take a~minimizing sequence
$\{\varphi_k\}$
for the functional~%
$\overline{I}$.
Then
\begin{equation*}
	\lim\limits_{k\rightarrow \infty} \overline{I}(\varphi_k)
	= \inf\limits_{\varphi\in\A} \overline{I}(\varphi).
\end{equation*}
By~\eqref{est:lower} and the assumption
$\inf\limits_{\varphi\in\A} \overline{I} (\varphi)<\infty$
we conclude that
the sequence
$(\varphi_k,\: \Adj D\varphi_k,\:J(\varphi_k))$
is bounded in the reflexive Banach space
$W^1_n(\Omega)\times L_{\frac{n}{n-1}}(\Omega)\times L_r(\Omega)$.
Consequently,
there exists a~subsequence
(which we also denote by
$\{(\varphi_k,\: \Adj D\varphi_k,\:J(\varphi_k))\}_{k\in \mathbb{N}}$)
weakly converging to an~element
$(\varphi_0,\: H,\: \delta)\in W^1_n(\Omega)\times L_{\frac{n}{n-1}}(\Omega)\times L_r(\Omega)$,
and furthermore, 
$H=\Adj D\varphi_0$
and
$\delta = J(\cdot,\varphi_0)$
by Lemma~\ref{lem:AdjJ}. 
Hence,
there exists a~minimizing sequence satisfying the conditions
\begin{equation}\label{weak_lim}
\begin{cases}
	\varphi_k \longrightarrow \varphi_0 &
	\text{ weakly in }  W^1_n(\Omega),
	\\
	\Adj D\varphi_k \longrightarrow \Adj D\varphi_0 &
	\text{ weakly in }  
	L_{\frac{n}{n-1}}(\Omega),
	\\
	J (\cdot,\varphi_k) \longrightarrow J (\cdot,\varphi_0) &
	\text{ weakly in } L_r(\Omega)
\end{cases}
\end{equation}
as
$k \rightarrow \infty$,
where
$\varphi_0$
guarantees the sharp lower bound 
$\overline{I} (\varphi_0)=\inf\limits_{\varphi\in\A} \overline{I} (\varphi)$.
It remains to verify that
$\varphi_0 \in \A$.
To this end,
we need the properties of mappings of class~%
$\A$
presented in the next subsection.

\subsubsection{Properties of admissible deformations
$\varphi\in\A$.} \label{subsec:prop_phi}

Let us state some properties of mappings of class~%
$\A$.

\begin{remark}\label{zam:phi} 
	If
	$\varphi \in \A$
	is a~homeomorphism
	then by Theorem~\ref{thm:Vod5}
	it induces a~bounded composition operator
	$\varphi^*: L^1_n(\Omega') \to L^1_q(\Omega)$,
	where
	$q=\frac{ns}{s+1}$
	and
	$\varphi^*(f)=f\circ \varphi$.
	Furthermore,
	we have the estimate
	\begin{equation}\label{est:thmV5}
   		\|\varphi^*\|\leq\|K_{\varphi,n}(\cdot) \mid L_{ns}(\Omega)\|
		\leq C \|\varphi^*\|
	\end{equation}
	with some constant~%
	$C$.
\end{remark}

\begin{remark}\label{zam:psi}
	If
	$\varphi \in \A$
	is a~homeomorphism
	then by Theorem~{\ref{thm:Vod6}} and {remark~\ref{zam:phi}}
	the inverse mapping
	$\psi = \varphi^{-1}$
	induces a~bounded operator
	$\psi^*: L^1_{q'}(\Omega) \to L^1_n(\Omega')$,
	where
	$q'=\frac{ns}{s-n+1}$,
	and %we have the estimate
	\begin{equation}\label{est:thmV6}
		\|\psi^*\| \leq \|K_{\psi,q'}(\cdot) \mid L_{\varrho}(\Omega')\|
		\leq \|K_{\varphi,n}(\cdot) \mid L_{ns}(\Omega)\|^{n-1},
	\end{equation}
	where
	$\frac{1}{\varrho}=\frac{1}{n}-\frac{1}{q'}=\frac{n-1}{ns}$.
\end{remark}

\begin{lemma}\label{lem:homeo}
	If
	$\overline{\varphi}: \overline{\Omega} \to \overline{\Omega'}$
	is a~homeomorphism
	then so is
	$\varphi:\overline{\Omega}\to\overline{\Omega'}$
	with
	$\varphi|_{\Omega}\in\A$.
\end{lemma}

\begin{proof}
The domain
$\Omega$
($\Omega'$)
has locally Lipschitz boundary, 
that is,
there exists a~tuple of charts
$\{\nu_j, \: U_j\}$
($\{\mu_k, \: V_k\}$),
where    
\begin{equation*}
	\nu_j: U_j\to B(0,r_j)\subset\R^n 
	\quad (\{\mu_k: V_k\to B'(0,r_k)\subset\R^n\})
\end{equation*}
and
\begin{align*}
	\nu_j(U_j \cap \Omega)  = B(0,r_j) & \cap \{ x_n > 0\}=O_j
	\:
	\text{and} \\%\quad
	& \nu_j (U_j \cap \Gamma) = B(0,r_j) \cap \{ x_n = 0\}
	\\ 
	(\mu_j (V_j \cap \Omega') = B'(0,r_k) & \cap \{ y_n > 0\}= O'_k
	\:
	\text{and}\\%\quad
	& \mu_k( V_k \cap \Gamma') = B'(0,r_k) \cap \{ y_n = 0\}),
\end{align*}
each mapping 
$\nu_k$ 
($\mu_k$)
is a~quasi-isometric mapping.

Then the mapping  
$\mu_k \circ \varphi \circ \nu_j^{-1}: O_{jk} \to O'_{jk}$,
where 
$O_{jk}=\nu_j\circ\varphi^{-1}(\mu_k^{-1}(O'_k))$ 
and 
$O'_{kj}=\mu_k\circ\varphi(\nu_j^{-1}(Q_j))$,
is of Sobolev class
$W^1_n ( O_{jk})$
(see~\cite[Theorem 1]{Resh2006}  for instance).

Consider the symmetrization
$\varphi_{jk,sym}: O_{jk} \cup \widetilde{O}_{jk} \to O'_{kj} \cup \widetilde{O}'_{kj}$
of the~mapping 
$\mu_k \circ  \varphi \circ \nu_j^{-1}$,
where
$\widetilde{O}_{jk} = \{ (x_1, x_2, \dots, -x_n) \mid (x_1, x_2, \dots, x_n) \in O_{jk} \}$
is the reflection of the set
$O_{jk}$
in the hyperplane
$O x_1 x_2 \dots x_{n-1}$
and
$\widetilde{O}'_{kj} = \{ (y_1, y_2, \dots, -y_n) \mid (y_1, y_2, \dots, y_n) \in O' \}$
is the reflection of the set
$O'_{kj}$
in the~hyperplane
$O' y_1 y_2 \dots y_{n-1}$,
\begin{equation*}
	\varphi_{jk,sym} (x_1, x_2, \dots, x_n) = 
	\begin{cases}
		\mu_k \circ \varphi \circ \nu_j^{-1} (x_1, x_2, \dots, x_n) 
		& \text{if } x_n \geq 0,
		\\
		\iota_n \circ \mu_k \circ \varphi \circ \nu_j^{-1} (x_1, x_2, \dots, -x_n) 
		& \text{if } x_n < 0,
	\end{cases}
\end{equation*}
where the mapping
$\iota_n: \R^n \to \R^n$
acts as 
$\iota_n(x_1, x_2, \dots, x_n) = (x_1, x_2, \dots, - x_n)$.

Since
$\varphi_{jk,sym} \in W^1_n (O_{jk}\cup \widetilde{O}_{jk})$,
by Theorem~\ref{thm:ManVill}
this mapping is either continuous, open and discrete,
or constant. In our case it is not a~constant.

Furthermore, we have the coincidence  
of mappings:
\begin{equation*}
	\varphi_{jk,sym}|_{U_j\cap\varphi^{-1}(V_k) \cap \Gamma} 
	= \overline{\varphi} |_{U_j \cap\varphi^{-1}(V_k)\cap \Gamma}.
\end{equation*}

Further, observe that
$\varphi$~%
and~%
$\overline{\varphi}$
are homotopic mappings
since they coincide on the boundary.
Consequently,
the degree
$\mu(\varphi,\Omega)$
of~%
$\varphi$
equals the degree
$\mu (\overline{\varphi}, \Omega)$
of~%
$\overline{\varphi}$,
and moreover,
$\mu (\overline{\varphi}, \Omega) = 1$.
Since~%
$\varphi$
is sense preserving by Lemma~\ref{lem:orient},
each point
$y \in \Omega'$
has exactly one preimage in~%
$\Omega$
(for more details,
see~\cite[Proposition 4.10]{Rickman1993});
therefore,
$\varphi$
is bijective.
The inverse mapping
$\varphi^{-1}$
is continuous
because~%
$\varphi$
is open
(by Theorem~\ref{thm:ManVill}).
Thus,
we conclude that~%
$\varphi$
is a~homeomorphism.
\end{proof}

\begin{lemma}\label{lem:prop2}
	Consider a~homeomorphism
	$\overline{\varphi}: \overline{\Omega} \to \overline{\Omega'}$
	and a~sequence
	$\{\varphi_k\}\subset\A$
	weakly converging to
	$\varphi$
	in
	$W^1_n(\Omega)$.
	If
	$\psi_k=\varphi^{-1}_k$
	then there exists a~subsequence
	$\{\psi_{k_l}\}$
	such that
	$\psi_{k_l}(y) \rightarrow \psi_0 (y)$
	uniformly up to the boundary.
\end{lemma}

\begin{proof}
Since each
$\varphi_k: \Omega \to \Omega'$
is a~homeomorphism by Lemma~\ref{lem:homeo}, 
it follows that
all mappings
$\psi_k=\varphi^{-1}_k$
are well-defined.

Since the domain 
$\Omega'$
is bounded
and
$\psi_k|_{\Gamma'}=\overline{\varphi}^{-1}|_{\Gamma'}$,
the sequence
$\psi_k$
is uniformly bounded.

On the other hand,
$f\in W^1_n(\R^n)$
satisfies the estimate
(corollary of Lemma~4.1 of~\cite{Mostow1968})
\begin{equation*}\label{est:mod_cont}
	{\rm osc}(f,\: S(y',r)) \leq L \biggl(\ln
	\frac{r_0}{r}\biggr)^{-\frac{1}{n}}
	\Biggl(\,\int\limits_{B(y',r_0)} |Df(y)|^n dy \Biggl)^{\frac{1}{n}},
\end{equation*}
where
$S(y',r)$
is the sphere of radius
$r<\dfrac{r_0}{2}$
centered at~%
$y'$
and
$B(y',r_0)$
is the ball of radius
$r_0$
centered at~%
$y'$.

Thus,
for
$\psi_k\in W^1_n(\Omega')$
we obtain
\begin{equation*}
	{\rm osc}(\psi_k,\: S(y',r)) \leq L \biggl(\ln
	\frac{r_0}{r}\biggr)^{-\frac{1}{n}} \Biggl(\,\int\limits_{B(y',r_0)} |D
	\psi_k(y)|^n dy \Biggl)^{\frac{1}{n}},
\end{equation*}
if
$B(y',r_0)\subset\Omega'$.
It follows the equicontinuity of the family of functions 
$\{\psi_k\}_{k\in\mathbb N}$ 
on any compact part of 
$\Omega'$.

Indeed,
H\"older's inequality and the estimate ~\eqref{est:thmV6}
yield
\begin{align*}
	\int\limits_{B(y',r_0)}&|D \psi_k (y)|^n dy \leq 
	\int\limits_{B(y',r_0)}\frac{|D \psi_k(y)|^n}{J(y,\psi_k)^{\frac{n}{q'}}} 
	J(y,\psi_k)^{\frac{n}{q'}} \,dy  
	\\ & 
	\leq\Biggl(\,\int\limits_{\Omega'} 
	\Biggl(\frac{|D\psi_k(y)|^n}{J(y,\psi_k)^{\frac{n}{q'}}}\Biggr)^{\frac{\varrho}{n}}
	\, dy \Biggr)^{\frac{n}{\varrho}} \Biggl(\,\int\limits_{B(y',r_0)}
	J(y,\psi_k)^{\frac{n}{q'}\cdot\frac{\varrho}{\varrho - n}}
	\,dy\Biggr)^{\frac{\varrho - n}{\varrho}}  
	\\ &
	\leq \|K_{\psi_k,q'}(\cdot) \mid L_{\varrho}(\Omega')\|^{n}  
	|\psi_k(B(y',r_0))|^{\frac{s}{s-(n-1)}}
	\\ &
	\leq \|K_{\varphi_k,n}(\cdot) \mid L_{ns}(\Omega)\|^{n(n-1)}  
	|\psi_k(B(y',r_0))|^{\frac{s}{s-(n-1)}}
	\\ &
	\leq \|M (\cdot) \mid L_{s}(\Omega)\|^{(n-1)} |\Omega|^{\frac{s}{s-(n-1)}}
\end{align*}
since
\begin{equation*}
	\frac{n}{q'}\cdot\frac{\varrho}{\varrho - n} =1 \quad
	\text{and} \quad 
	\frac{\varrho - n}{\varrho} = \frac{s}{s-(n-1)}.
\end{equation*}

To show the equicontinuity of the family of functions 
$\{\psi_k\}_{k\in\mathbb N}$ 
near 
$\partial\Omega'$ 
we use 
the method introduced in papers~%
\cite{GV1978, Vod1975, VGR1979}.
We fix an arbitrary ball
$B\Subset\Omega$  
such that for all 
$k\geq k_0$, 
where  
$k_0$ 
is big enough, the inclusion  
$\varphi_k(B)\subset B'\Subset\Omega' $ 
holds 
(here 
$B'$ 
is a~ball).  
Consider two points 
$x$,
$y\in \Omega'\setminus  B'$ 
and an arbitrary curve 
$\gamma:[0,1]\to\Omega'\setminus  B'$  
with endpoints  
$x$ 
and 
$y$.

\begin{definition}
	The  quantity
	\begin{equation*}
		\operatorname{cap}(B',\gamma; L^1_n(\Omega')) 
		= \inf\limits_u\int\limits_{\Omega'}|D u|^n\,dy, 
	\end{equation*}
	where the lower bound is taken over all continuous functions
 	$u\in L^1_n(\Omega')$
	such that  
	$u=0$ 
	on 
	$B'$ 
	and 
	$u\geq 1$ 
	on 
	$\gamma$,
	is called \textit{capacity} of the pair sets 
	$(B',\gamma)$.
\end{definition}

If we now define
\begin{equation*}
	d'_n(x,y)=\inf\limits_\gamma\operatorname{cap}(B',\gamma; L^1_n(\Omega'))
\end{equation*}
where the lower bound is taken over all curves
$\gamma:[0,1]\to\Omega'\setminus  B'$ 
with endpoints 
$x$ 
and 
$y$,  
we obtain a~metric on the set
$\Omega'\setminus  B'$~\cite{GV1978, VGR1979}.

By Remark~\ref{zam:phi} 
the composition operator
$\varphi_k^*: L^1_n(\Omega') \to L^1_q(\Omega)$,
where
$q=\frac{ns}{s+1}$
and
$\varphi_k^*(f)=f\circ \varphi_k$
is bounded and by Lemma~\ref{lem:bounded_comp} below
 $\|\varphi_k^*\|\leq \|M(\cdot) \mid L_{s}(\Omega)\|^\frac{1}{n}$.

 From here and~\eqref{est:thmV5} we infer 
\begin{equation*}
	\operatorname{cap}(B,\varphi^{-1}_k(\gamma); L^1_q(\Omega))^\frac{1}{q}\leq 
	\|M(\cdot) \mid L_{s}(\Omega)\|^\frac{1}{n}
	\operatorname{cap}(B',\gamma; L^1_n(\Omega'))^\frac{1}{n}.
\end{equation*}

The last estimate implies 
\begin{equation*}
	d_q(\psi_k(x),\psi_k(y))\leq \|M(\cdot) \mid L_{s}(\Omega)\|^\frac{q}{n} 
	d'_n(x,y)^\frac{q}{n}.
\end{equation*}
As soon as  domain 
$\Omega'$ 
($\Omega$) 
meets 
$(\epsilon,\delta)$-condition 
in the sense of paper 
\cite{Jones1981}
there exists an extension  operator 
$i':L^1_n(\Omega')\to L^1_n(\mathbb R^n)$
($i':L^1_q(\Omega)\to L^1_q(\mathbb R^n)$). 
It follows that the topology of the metric space 
$(\Omega'\setminus  B',d'_n)$ 
(\,$(\Omega\setminus  B,d_q)$\,)
is equivalent to the Euclidean one.
The last property is a~consequence of the following  two assertions:
\begin{enumerate}
	\item 
		\textit{given 
		$\rho>0$, 
		there exist 
		$\tau>0$ 
		such that 
		$|x-y|<\rho$ 
		provided 
		$d_q(x,y)<\tau$ 
		for points 
		$x$, 
		$y\in \Omega$ 
		closed enough to 
		$\partial \Omega$};
	\item 
		\textit{given 
		$\tau>0$, 
		there exist 
		$\varkappa>0$ 
		such that 
		$d'_n(x,y)<\tau$ 
		provided 
		$|x-y|<\varkappa$ 
		for points 
		$x$,
		$y\in \Omega'$ 
		closed enough to 
		$\partial \Omega'$}.
\end{enumerate}

Thus,
we see that
the family
$\{\psi_k,\: k\in\mathbb{N}\}$
is equicontinuous and uniformly bounded. 
By the Arzel\`a--Ascoli theorem
(see~\cite{KolmFom1976} for instance)
there exists a~subsequence 
$\{\psi_{k_l}\}$
converging uniformly to a~mapping
$\psi_0$
as
$k_l \rightarrow \infty$.
\end{proof}

\subsubsection{Nonnegativity of the Jacobian.} \label{subsec:J>0}

\begin{lemma}
	The limit mapping
	$\varphi_0$
	satisfies
	$J(\cdot,\varphi_0)\geq 0$
	a.~e. in 
	$\Omega$.
\end{lemma}

\begin{proof}
The inequality
$J(\cdot,\varphi_0)\geq 0$
follows directly from the weak convergence of
$J(\cdot,\varphi_k)$
in
$L_r(\Omega)$
with
$r>1$.

Even in the case
$r=1$
we can establish the nonnegativity of the Jacobian
by using weak convergence %of the Jacobians
(Theorem~{\ref{thm:Resh}}).
Indeed,
the sequence
$\varphi_k \in W^1_n(\Omega)$
is bounded in
$W^1_n(\Omega)$,
and,
since the embedding of
$W^1_n(\Omega)$
into
$L_1(\Omega)$
is compact
(Theorem~\ref{thm:embedding}), 
there exists a~subsequence converging in
$L_1(\Omega)$.
Hence,
for every continuous function 
$f:\Omega \to \R$
with compact support in~%
$\Omega$
we have
\begin{equation*}
	\int\limits_{\Omega} f(x)J(x,\varphi_k)\,dx \xrightarrow[k \rightarrow \infty]{} 
	\int\limits_{\Omega} f(x)J(x,\varphi_0)\,dx.
\end{equation*}
It follows immediately  
\begin{equation*}
	\int\limits_{\Omega} f(x)J(x,\varphi_0)\,dx\geq 0
\end{equation*}
for an arbitrary function 
$f(x)\geq 0$.
Hence we imply 
$J(x,\varphi_0)\,dx\geq 0$ 
a.~e. 
\end{proof}

%========

\subsubsection {Behavior on the boundary.} \label{subsec:phi_0_Gamma}

\begin{lemma}
	Equality
	$\varphi_0|_\Gamma=\overline{\varphi}|_\Gamma$
	holds
	a.~e.~in~%
	$\Gamma$.
\end{lemma}

\begin{proof}
Since the trace operator is compact,
for every
$1< q < \infty$
we deduce
\begin{equation*}
	\varphi_k \longrightarrow \varphi_0
	\text{ weakly in }
	W^1_n(\Omega)\:\Rightarrow\: \tr\varphi_k\longrightarrow\tr \varphi_0
	\text{ in } L_p(\Gamma).
\end{equation*}
Extracting a~subsequence converging almost everywhere in~%
$\Gamma$,
we obtain
$\varphi_0|_\Gamma=\overline{\varphi}|_\Gamma$
almost everywhere in~%
$\Gamma$.
\end{proof}

\subsubsection{Boundedness of the composition operator.} \label{subsec:bounded_comp}

\begin{lemma} \label{lem:bounded_comp}
	The mapping
	$\varphi_0$
	induces a~bounded composition operator
	$\varphi_0^*: L^1_n(\Omega')\cap {\rm Lip}(\Omega')
	\to L^1_q(\Omega)$.
\end{lemma}

\begin{proof}
Consider
$u\in L^1_{n}(\Omega')\cap {\rm Lip}(\Omega')$.
Since~\eqref{est:thmV5} yields
\begin{equation*}
	\|\varphi_k^*\|\leq \|K_{\varphi_k,n}(\cdot) \mid L_{ns}(\Omega)\| \leq 
	\|M(\cdot) \mid L_{s}(\Omega)\|^\frac{1}{n},
\end{equation*}
the sequence
$w_k=\varphi_k^*u=u\circ \varphi_k$
is bounded in
$L^1_q(\Omega)$.
Using a~compact embedding into the Sobolev space,
we obtain a~subsequence with
$w_k\rightarrow w_0$
in
$L_r(\Omega)$,
where
$r\leq \frac{ns}{ns-s-1}$.
From this sequence,
in turn,
we can extract a~subsequence converging almost everywhere in~%
$\Omega$.
If
$u\in L^1_{n}(\Omega')\cap {\rm Lip}(\Omega')$
then
$w_0(x)=u\circ\varphi_0(x)$
for almost all
$x\in\Omega$.

On the other hand,
since
$w_k$
converges weakly to
$w_0$
in
$L^1_q (\Omega)$,
we have 
\begin{multline*}
	\|u\circ \varphi_0 \mid L^1_q(\Omega)\|=\|w_0 \mid L^1_q(\Omega)\|\leq 
	\varliminf \limits_{k\rightarrow\infty} \|w_k \mid L^1_q(\Omega)\| 
	\\ = 
	\varliminf \limits_{k\rightarrow\infty} \|\varphi^*_k(u) \mid L^1_q(\Omega)\|
	\leq \varliminf\limits_{k\rightarrow\infty} \|\varphi_k^*\| \cdot \|u\mid L^1_n(\Omega')\|
	\\ \leq 
	\|M(\cdot) \mid L_{s}(\Omega)\|^\frac{1}{n} \cdot \|u \mid L^1_n(\Omega')\|.
\end{multline*}

Thus,
$\varphi_0$
induces a~bounded composition operator
$\varphi_0^*: L^1_n(\Omega')\cap {\rm Lip}(\Omega')
\to L^1_q(\Omega)$,
and moreover,
$\|\varphi_0^*\|\leq \|M(\cdot) \mid L_{s}(\Omega)\|^\frac{1}{n}$.
\end{proof}

\subsubsection{Injectivity.}\label{subsec:inj_phi}

Verify that
the mapping
$\varphi_0: \Omega \to \overline{\Omega'}$
is injective almost everywhere
(since~%
$\varphi_0$
is the pointwise limit of the homeomorphisms
$\varphi_k: \Omega \to \Omega'$,
the images of some points
$x\in \Omega$
may lie on the boundary
$\partial \Omega'$).
Recall the definition.

\begin{definition}
	A~mapping
	$\varphi: \Omega \to \overline{\Omega'}$
	is called {\it injective almost everywhere}
	whenever there exists a~negligible set~%
	$S$
	outside which~%
	$\varphi$
	is injective.
\end{definition}

Denote by
$S \subset\Omega$
a~negligible set
on which the convergence 
$\varphi_k(x)\rightarrow \varphi_0(x)$
as
$k \rightarrow \infty$
fails.
If
$x\in\Omega\setminus S$
with %such that
$\varphi(x) \in \Omega'$
then the
injectivity follows from the uniform convergence of
$\psi_k$
on~%
$\Omega'$ (see Lemma~\ref{lem:prop2})
and the equality
\begin{equation*}
	\psi_k \circ \varphi_k(x)=x.
\end{equation*}
Passing to the limit as
$k \rightarrow \infty$,
we infer that
\begin{equation*}
	\psi_0 \circ \varphi_0(x)=x,\: x\in\Omega\setminus S.
\end{equation*}
Hence,
we deduce that
if
$\varphi_0(x_1)=\varphi_0(x_2)\in  \Omega'$
for two points
$x_1,x_2\in\Omega\setminus S$
then
$x_1=x_2$.

It remains to verify that
the set of points
$x\in\Omega$
with
$\varphi(x) \in \partial\Omega'$
is negligible.
The argument rests on the method of proof of~\cite[Theorem 4]{VodUhl2002}.
For the reader's convenience,
we present here the new details of this method. 

Given a~bounded open set
$A'\subset\R^n$,
define the class of functions
$\overset{\circ}L{}^1_p (A')$
as the closure of the subspace
$C_0^\infty (A')$
in the seminorm of
$L^1_p (A')$.
In general,
a~function
$f \in \overset{\circ}L{}^1_p (A')$
is defined only on the set~%
$A'$,
but,
extending it by zero,
we may assume that
$f \in L^1_p (\R^n)$.

\begin{lemma}[cf.~Lemma~1 of~\cite{VodUhl2002}]\label{lem:Phi}
	Assume that
	the mapping
	$\varphi: \Omega \to \overline{\Omega'}$
	induces a~bounded composition operator
	\begin{equation*}
		\varphi^*: L^1_p (\Omega') \cap {\rm Lip}(\Omega') \to L^1_q (\Omega), 
		\quad 1 \leq q < p \leq \infty.
	\end{equation*}
	Then 
	\begin{equation*}
		\Phi(A') = \sup \limits_{f \in \overset{\circ}L{}^1_p (A') \cap {\rm Lip}(A')} 
		\Bigg( \frac{\|\varphi^* f \mid L^1_q(\Omega) \|}
		{\|f \mid  L^1_p (A' \cap \Omega') \|} \Bigg)^{\sigma}, \quad
		\sigma = 
		\begin{cases}
			\frac{pq}{p-q} & \quad
			\text{for } p<\infty,
			\\
			q &  \quad
			\text{for } p=\infty,
	\end{cases}
\end{equation*}
is a~bounded monotone countably additive function
defined on the open bounded sets
$A'$
with
$A' \cap \Omega' \neq \emptyset$.
\end{lemma}

\begin{proof}
It is obvious that
$\Phi(A'_1)\leq \Phi(A'_2)$
whenever
$A'_1 \subset A'_2$.

Take disjoint sets
$A'_i$, 
$i\in \mathbb{N}$
in 
$\Omega'$
and put
$A'_0 = \bigcup\limits_{i=1}^{\infty} A'_i$.
Consider a~function
$f_i \in \overset{\circ}L{}^1_p (A'_i) \cap {\rm Lip}(A_i')$
such that
the conditions
\begin{equation*}
	\| \varphi^* f_i \mid L^1_q(\Omega) \| 
	\geq (\Phi(A'_i) (1- \frac{\varepsilon}{2^i}))^{1/\sigma}
	\| f_i \mid \overset{\circ}L{}^1_p (A'_i) \|
\end{equation*}
and 
\begin{align*}
	& \| f_i \mid \overset{\circ}L{}^1_p (A'_i) \|^p = \Phi(A'_i) (1- \frac{\varepsilon}{2^i})
	\text{ for } p < \infty\\
	& (\| f_i \mid \overset{\circ}L{}^1_p (A'_i) \|^p = 1
	\text{ for } p=\infty)
\end{align*}
hold simultaneously,
where
$0< \varepsilon <1$.
Putting
$f_N = \sum \limits_{i=1} \limits ^{N} f_i \in L^1_p(\Omega') \cap {\rm Lip}(\Omega')$,
and applying H\"older's inequality
(the case of equality),
we obtain
\begin{multline*}
	\|\varphi^* f_N \mid L^1_q(\Omega) \| \geq 
	\bigg( \sum \limits_{i=1}\limits^{N}
	\Big(\Phi(A'_i) \Big(1- \frac{\varepsilon}{2^i}\Big)\Big)^{\frac{q}{\sigma}}
	\big\| f_i | \overset{\circ}L{}^1_p (A'_i) \big\| ^q\bigg)^{\frac{1}{q}}
	\\ = 
	\bigg( \sum \limits_{i=1}\limits^{N}
	\Phi(A'_i) \Big(1- \frac{\varepsilon}{2^i}\Big)\bigg)^{\frac{1}{\sigma}}
	\bigg\| f_N \mid \overset{\circ}L{}^1_p \Big(\bigcup\limits_{i=1}\limits^{N} A'_i\Big) 
	\bigg\|
	\\ \geq 
	\bigg( \sum \limits_{i=1}\limits^{N}
	\Phi(A'_i) - \varepsilon \Phi(A'_0) \bigg)^{\frac{1}{\sigma}}
	\bigg\| f_N \mid \overset{\circ}L{}^1_p \Big(\bigcup\limits_{i=1}\limits^{N} A'_i\Big) \bigg\|
\end{multline*}
since the set
$A_i$,
on which the functions
$\nabla \varphi^* f_i$
are nonvanishing,
are disjoint.
This implies that
\begin{equation*}
	\Phi(A'_0)^{\frac{1}{\sigma}} \geq \sup 
	\frac{\| \varphi^* f_N  \mid L^1_p (\Omega) \|}
	{\bigg\| f_N \mid \overset{\circ}L{}^1_p \Big(\bigcup\limits_{i=1}\limits^{N} A'_i\Big) 
	\bigg\|}
	\\ \geq 
	\bigg( \sum \limits_{i=1}\limits^{N}
	\Phi(A'_i) - \varepsilon \Phi(A'_0) \bigg)^{\frac{1}{\sigma}},
\end{equation*}
where we take the sharp upper bound over all functions
\begin{equation*}
	f_N \in \overset{\circ}L{}^1_p 
	\Big(\bigcup\limits_{i=1}\limits^{N} A'_i\Big) \cap {\rm Lip}
	\Big(\bigcup\limits_{i=1}\limits^{N} A'_i\Big),
	\quad
	f_N = \sum \limits_{i=1} \limits ^{N} f_i,
\end{equation*}
and
$f_i$
are of the form indicated above.
Since
$N$~%
and~%
$\varepsilon$
are arbitrary,
%we have
\begin{equation*}
	\sum\limits_{i=1}\limits^{\infty} \Phi(A'_i) \leq \Phi 
	\Big(\bigcup\limits_{i=1}\limits^{\infty} A'_i\Big).
\end{equation*}

We can verify the inverse inequality directly
by using the definition of~% the function
$\Phi$.
\end{proof}

\begin{lemma}
	Take a~monotone countably additive function~%
	$\Phi$
	defined on the~bounded open sets
	$A'$
	with
	$A' \cap \Omega' \neq \emptyset$.
	For every set 
	$A'$
	there exists a~sequence of balls
	$\{B_j\}$
	such that
	\begin{enumerate}
	\item
		the families of
		$\{B_j\}$
		and
		$\{2 B_j\}$
		constitute finite coverings of~% the set
		$U$;
	
	\item
		$\sum\limits_{j=1}\limits^{\infty} \Phi(2 B_j) \leq \zeta_n \Phi (U)$,
		where the constant
		$\zeta_n$
		depends only on the dimension~%
		$n$.
	\end{enumerate}
\end{lemma}

\begin{proof}
In accordance with Lemma~\ref{lem:bez}
construct two sequences
$\{B_j\}$
and
$\{2 B_j\}$
of balls
and subdivide the latter
into
$\zeta_n$
subfamilies
$\{2 B_{1j}\}_{j=1}^{\infty}, \dots, \{2 B_{\zeta_n j}\}_{j=1}^{\infty}$
so that
in each tuple the balls are disjoint:
$2 B_{ki} \cap 2 B_{kj} = \emptyset$
for
$i \neq j$
and
$k = 1, \dots, \zeta_n$.
Consequently,
\begin{equation*}
	\sum\limits_{j=1}\limits^{\infty} \Phi(2 B_j) = \sum\limits_{k=1}
	\limits^{\zeta_n}
	\sum\limits_{j=1}\limits^{\infty} \Phi(2 B_{kj}) \leq
	\sum\limits_{k=1}\limits^{\zeta_n} \Phi(U) = 
	\zeta_n \Phi(U).
\end{equation*}
\end{proof}

\begin{lemma}\label{lem:N-1}
	If a~measurable almost everywhere injective mapping  
	$\varphi: \Omega \to \overline{\Omega'}$
	induces a~bounded composition operator
	\begin{equation*}
		\varphi^*: L^1_p (\Omega') \cap {\rm Lip}(\Omega') \to 
		L^1_q (\Omega), 
		\quad 1 \leq q < p \leq n,
	\end{equation*}
	then
	$|\varphi^{-1}(E)| = 0$
	for every set
	$E \subset \Gamma'= \partial \Omega'$.
\end{lemma}

\begin{proof}
Consider the cutoff 
$\eta \in C_0^{\infty}(\R^n)$
equal to~1 on
$B(0,1)$
and vanishing outside
$B(0,2)$.
By Lemma~\ref{lem:Phi}
the function
$f (y) = \eta \big( \frac{y-y_0}{r} \big)$
satisfies %the estimate 
\begin{equation*}
	\| \varphi^* f \mid L^1_q (\Omega) \| 
	\leq C_1 \Phi(2B)^{\frac{1}{\sigma}} |B|^{\frac{1}{p}- \frac{1}{n}},
\end{equation*} 
where
$B \cap \Omega' \neq \emptyset$.
Take an~set 
$E \subset \Gamma'$
with
$|E| = 0$.
Since~%
$\varphi$
is a~mapping with finite distortion,
$\varphi^{-1}(E) \neq \Omega$
(otherwise,
$J(x,\varphi)=0$
and,
consequently,
$D\varphi(x) = 0$,
that is,~%
$\varphi$
is a~constant mapping).
Hence,
there is a~cube
$Q \subset \Omega$
such that
$2Q \subset \Omega$
and
$| Q \setminus \varphi^{-1} (E) | >0$
(here
$2Q$
is a~cube with the same center as~%
$Q$
and the edges stretched by a~factor of two compared to~%
$Q$).
Since~%
$\varphi$
is a~measurable mapping,
by Luzin's theorem  there is a~compact set
$T \subset Q \setminus \varphi^{-1} (E)$
of positive measure such that
$\varphi: T \to \Omega'$
is continuous.
Then,
the image
$\varphi(T) \subset \Omega'$
is compact and
$\varphi(T) \cap E = \emptyset$.
Consider an~open set 
$U \supset E$
with
$\varphi(T)\cap U = \emptyset$
and
$U\cap \Omega' \neq \emptyset$.
Choose a~tuple
$\{B(y_i, r_i)\}$
of balls in accordance with Lemma~\ref{lem:bez}:
$\{B(y_i, r_i)\}$
and
$\{B(y_i, 2 r_i)\}$
are coverings of~% the set
$U$,
and the multiplicity of the covering
$\{B(y_i, 2 r_i)\}$
is finite
($B(y_i, 2 r_i) \subset U$
for all~%
$i \in \mathbb{N}$).
Then the function
$f_i$
associated to the ball
$B(y_i, r_i)$
satisfies
$\varphi^*f_i = 1$
on %the set
$\varphi^{-1}(B(y_i, r_i))$
and
$\varphi^* f = 0$
outside %the set
$\varphi^{-1}(B(y_i, 2 r_i))$,
In~particular,
$\varphi^* f_i = 0$
on~% the set
$T$.
In addition,
we have the estimate
\begin{equation*}
	\| \varphi^* f_i \mid L^1_q (2Q) \| \leq \| \varphi^* f_i \mid L^1_q (\Omega)  \| 
	\leq 
	C_1 \Phi(B(y_i, 2 r_i))^{\frac{1}{\sigma}} |B(y_i, r_i)|^{\frac{1}{p}- \frac{1}{n}}.
\end{equation*} 
By Poincar\'e inequality
(see~\cite{Maz2011} for instance),
for every function
$g \in W^1_{q,\loc} (Q)$,
where
$q<n$,
vanishing on~% the set
$T$
we have
\begin{equation*}
	\bigg( \int\limits_Q | g |^{q^*} \,dx \bigg)^{1/q^*} 
	\leq 
	C_2 l(Q)^{n/q^*}\bigg( \int\limits_{2Q} |\nabla g |^{q} \,dx \bigg)^{1/q},
\end{equation*} 
where
$q^* = \frac{nq}{n-q}$
and
$l(Q)$
is the edge length of~%the cube
$Q$.

Applying Poincar\'e inequality to the function
$\varphi^* f_i$
and using the last two estimates,
we obtain
\begin{equation*}
	|\varphi^{-1}(B(y_i, r_i)) \cap Q|^{\frac{1}{q} - \frac{1}{n}}  
	\leq
	C_3 \Phi(B(y_i, 2 r_i))^{\frac{1}{\sigma}} 
	|B(y_i, r_i)|^{\frac{1}{p}- \frac{1}{n}}.
\end{equation*}
In turn,
H\"older's inequality guarantees that %the relation
\begin{multline*}
	\Bigg(\sum\limits_{i=1}\limits^{\infty}|\varphi^{-1}(B(y_i, r_i)) \cap Q|\Bigg)
	^{\frac{1}{q} - \frac{1}{n}}
	\\ \leq
	C_3 \Bigg(\sum\limits_{i=1}\limits^{\infty}\Phi(B(y_i, 2 r_i))\Bigg)^{\frac{1}{\sigma}}% \cdot 
	\Bigg(\sum\limits_{i=1}\limits^{\infty}|B(y_i, r_i)|\Bigg)^{\frac{1}{p}- \frac{1}{n}}.
\end{multline*} 
As the open set~%
$U$
is arbitrary,
this estimate yields
$|\varphi^{-1}(E) \cap Q|=0$.
Since the cube
$Q \subset \Omega$
is arbitrary,
it follows that
$|\varphi^{-1}(E)|=0$.
\end{proof}

Since for the domain
$\Omega'$
with locally Lipschitz boundary
we have 
${|\partial \Omega'| = 0}$,
Lemmas~\ref{lem:bounded_comp}~and~\ref{lem:N-1} implies the next lemma.

\begin{lemma}
	Consider a~mapping
	$\varphi_0: \Omega \to \Omega'$,
	where 
	$\Omega$, 
	$\Omega' \subset \R^n$
	are domains with locally Lipschitz boundary, 
	with
	$\varphi_0 \in W^1_n(\Omega)$
	and
	$J(x,\varphi_0) \geq 0$
	a.~e.~in
	$\Omega$,
	such that 
	\begin{enumerate}
	\item
		there exists a~sequence of homeomorphisms
		$\varphi_k \in W^1_n(\Omega) \cap FD(\Omega)$
		with
		$J(x,\varphi_k) \geq 0$
		almost everywhere in~%
		$\Omega$,
		such that
		$\varphi_k \rightarrow \varphi_0$
		weakly in~%
		$W^1_n(\Omega)$;

	\item
		each mapping
		$\varphi_k$
		induces a~bounded composition operator
		$\varphi_k^*: L^1_n(\Omega') \to L^1_q(\Omega)$
		with
		$q=\frac{ns}{s+1}$,
		where
		$\varphi_k^*(f)=f\circ \varphi_k$;

	\item
		the norms of the operators
		$\| \varphi_k^* \|$
		are jointly bounded;

	\item
		$\varphi_k\vert_{\partial \Omega} = \varphi_0\vert_{\partial \Omega}$.
	\end{enumerate}
	Then the mapping
	$\varphi_0$
	is injective almost everywhere.
\end{lemma}

Let us mention another interesting corollary of Theorem~\ref{thm4}.

\begin{lemma}\label{lem:J>0}
	If an~almost everywhere injective mapping
	$\varphi: \Omega \to \Omega'$
	with
	$\varphi \in W^1_n(\Omega)$
	and
	$J(x,\varphi) \geq 0$
	a.~e.~in
	$\Omega$
	has Luzin $\N^{-1}$-property
	then
	${J(x,\varphi) > 0}$
	for almost all
	$x \in \Omega$.
\end{lemma}

\begin{proof}
Denote by~%
$E$
a~set outside which the mapping~%
$\varphi$
is approximatively differentiable and has  Luzin $\N^{-1}$-property. 
Since
$\varphi \in W^1_n(\Omega)$,
it follows that
$|E| = 0$
(see~\cite{Whit1951, Haj1993}).
In addition,
we may assume that
\begin{equation*}
	Z = \{ x\in \Omega \setminus E \mid J(x, \varphi) = 0\}
\end{equation*}
is a~Borel set. 
Put
$\sigma = \varphi(Z)$.
By the change-of-variable formula
(Theorem~\ref{thm:change_var}),
taking the injectivity of~%
$\varphi$
into account,
we obtain
\begin{equation*}
	\int\limits_{\Omega \setminus \Sigma} \chi_{Z} (x) J (x, \varphi) \, dx
	= \int\limits_{\Omega \setminus \Sigma} (\chi_{\sigma} \circ \varphi)(x) J (x, \varphi) \, dx 
	= \int\limits_{\Omega'} \chi_{\sigma} (y) \, dy. 
\end{equation*}
By construction,
the expression in the left-hand side vanishes;
consequently,
$|\sigma| = 0$.
On the other hand,
since~%
$\varphi$
has  Luzin $\N^{-1}$-property,
we have
$|Z| = 0$.
\end{proof}

Using Lemma~\ref{lem:J>0},
we conclude that
the limit mapping
$\varphi_0$
satisfies the~strict inequality
$J(x,\varphi_0) > 0$
a.~e.~in~%
$\Omega$.

\subsubsection{Behavior of the distortion coefficient.}\label{subsec:prop_K}

The next lemma follows directly from~\cite[Theorem 1]{VodUhl2002}.

\begin{lemma}
	If an~almost everywhere injective mapping 
	$\varphi: \Omega \to \Omega'$
	generates a~bounded composition operator
	${\varphi^*: L^1_n(\Omega') \cap {\rm Lip}(\Omega') 
	\to L^1_q(\Omega)}$
	with
	$q=\frac{ns}{s+1}$
	then 
	\begin{enumerate}
	\item
		$\varphi\in ACL(\Omega)$;
	\item
		$\varphi$
		has finite distortion;
	\item
		$\|K_{\varphi,n}(\cdot) \mid L_{ns}(\Omega)\|\leq \widetilde{M}<\infty$.
	\end{enumerate}	
\end{lemma}

Indeed,
for each almost everywhere injective mapping~%
$\varphi$
the function
$H_q(y)$
of~\cite{VodUhl2002}
becomes quite simple:
\begin{equation*}
	H_q(y) = 
	\begin{cases}
		\dfrac{|D \varphi (x)|}{|J(x, \varphi)|^{1/q}} & \quad
		\text{if } 
		y = \varphi(x), \: x\in \Omega \setminus (Z \cup \Sigma \cup I),\\
		0 & \quad
		\text{otherwise}.
	\end{cases}
\end{equation*}

The necessary relations follow from the change-of-variable formula
and the inequality
$J(x, \varphi) \geq 0$:
\begin{multline*}
	H_{p,q} (\Omega')^{\varkappa} 
	= \| H_q (\cdot)\mid L_\varkappa(\Omega') \|^{\varkappa} = 
	\int\limits_{\Omega'} ( H_q(y))^{\frac{pq}{p-q}} \, dy
	\\ = 
	\int\limits_{\Omega \setminus (Z \cup \Sigma \cup I)} 
	\Biggl( \frac{|D \varphi (x)|}{|J(x, \varphi)|^{1/q}} \Biggr)^{\frac{p}{p-q}} 
	J(x, \varphi) \, dx = 
	\int\limits_{\Omega \setminus (Z \cup \Sigma \cup I)} 
	\frac{|D \varphi (x)|^{\frac{pq}{p-q}}}{|J(x, \varphi)|^{\frac{q}{p-q}}} \, dx 
	\\ =
	\int\limits_{\Omega \setminus (Z \cup \Sigma \cup I)} 
	\frac{|D \varphi (x)|^{ns}}{|J(x, \varphi)|^{s}} \, dx = 
	\|K_{\varphi,n}(\cdot) \mid L_{ns}(\Omega)\|^{ns}.
\end{multline*}

By Theorem~%
\ref{thm:ManVill},
the mapping
$\varphi_0 \in W^1_n(\Omega)$
is continuous, open, and discrete;
furthermore,
$\varphi_0|_{\Gamma}=\overline{\varphi}|_{\Gamma}$,
and so
$\varphi_0$
is a~homeomorphism
by Lemma~\ref{lem:homeo}.
It remains to verify that
the pointwise inequality
\begin{equation*}
	\frac{|D\varphi_0(x)|^n}{J(x,\varphi_0)}\leq M (x)
\end{equation*}
holds almost everywhere in~%
$\Omega$.

\begin{lemma}
	Consider a~sequence
	$\{\varphi_k\}_{k\in \mathbb{N}}$
	of mappings
	$\varphi_k: \Omega \to \Omega'$
	with finite distortion
	weakly converging in
	$W_n^1(\Omega)$
	to a~mapping
	$\varphi_0: \Omega \to \Omega'$.
	Assume that
	$J(x,\varphi_k)\geq 0$
	almost everywhere in~%
	$\Omega$.
	Suppose also that
	there exists an~almost everywhere nonnegative function 
	$M (x) \in L_{s}(\Omega)$
	such that
	\begin{equation*}%\label{est:lem_phi_k}
		K_{\varphi_k,n}(x) \leq M(x)^{\frac{1}{n}} \:
		\text{for all } k\in \mathbb{N}, \:
		\text{for almost all } x \in \Omega.
	\end{equation*}
	Then the limit mapping~%
	$\varphi$
	satisfies
	\begin{equation*}
		K_{\varphi_0,n}(x) \leq M(x)^{\frac{1}{n}} \:
		\text{for almost all } x \in \Omega.
	\end{equation*}
\end{lemma}

\begin{proof}
Take a~test function
$\theta \in C_0^{\infty}(\Omega)$.
H\"older's inequality yields
\begin{multline*}
	\int\limits_{\Omega} |D\varphi_j (x)|^{\frac{ns}{s+1}} \theta(x) \,dx =
	\int\limits_{\Omega} \frac{ |D\varphi_j (x)|^{\frac{ns}{s+1}}}
	{J(x, \varphi_j)^{\frac{s}{s+1}}} J(x, \varphi_j)^{\frac{s}{s+1}} \theta(x) \,dx
	\\ \leq
	\Biggl( \int\limits_{\Omega} \frac{ |D\varphi_j (x)|^{\frac{ns}{s+1} (s+1)}}
	{J(x, \varphi_j)^{\frac{s}{s+1} (s+1)}} \theta(x)^{\frac{1}{s+1} (s+1)} \,dx\Biggr)
	^{\frac{1}{s+1}}
	\\ \times 
	\Biggl( \int\limits_{\Omega} J(x, \varphi_j)^{\frac{s}{s+1} \frac{s+1}{s}} 
	\theta(x)^{\frac{s}{s+1}\frac{s+1}{s}} \,dx \Biggr)^{\frac{s}{s+1}}\\ \leq
	\Biggl( \int\limits_{\Omega} K_{\varphi_j,n}^{ns}(x)  \theta(x) \,dx \Biggr)
	^{\frac{1}{s+1}}
	\Biggl( \int\limits_{\Omega} J(x, \varphi_j) \theta(x)\,dx \Biggr)^{\frac{s}{s+1}}
	\\ \leq 
	\Biggl( \int\limits_{\Omega} M^{s}(x)  \theta(x) \,dx \Biggr)^{\frac{1}{s+1}}
	\Biggl( \int\limits_{\Omega} J(x, \varphi_j) \theta(x)\,dx \Biggr)^{\frac{s}{s+1}}.
\end{multline*}
By lower semicontinuity,
\cite[Ch. 3,~\textsection~\,3]{Resh1982}, 
we can estimate the left-hand side: 
\begin{equation*}
	\int\limits_{\Omega} |D\varphi_0 (x)|^{\frac{ns}{s+1}} \theta(x) \,dx 
	\leq
	\varliminf\limits_{j\rightarrow \infty}\int\limits_{\Omega} 
	|D\varphi_j (x)|^{\frac{ns}{s+1}} \theta(x) \,dx.
\end{equation*}
On the other hand,
Theorem~{\ref{thm:Resh}} yields
\begin{equation*}
	\int\limits_{\Omega} J(x, \varphi_j) \theta(x)\,dx \xrightarrow[j \rightarrow \infty]{}  
	\int\limits_{\Omega} J(x, \varphi_0) \theta(x)\,dx
\end{equation*}
for every function 
$\theta \in C_0(\Omega)$.
Finally,
we obtain the inequality
\begin{multline}\label{neq:semicont}
	\int\limits_{\Omega} |D\varphi_0 (x)|^{\frac{ns}{s+1}} \theta(x) \,dx
	\\ \leq
	\Biggl( \int\limits_{\Omega} M^{s}(x)  \theta(x) \,dx \Biggr)^{\frac{1}{s+1}} 
	\Biggl( \int\limits_{\Omega} J(x, \varphi_0) \theta(x)\,dx \Biggr)^{\frac{s}{s+1}}.
\end{multline}

Consider the family of functions
$\theta_{r,\varepsilon,y} \in C_0^{\infty}(\Omega)$:	
\begin{equation*}%\label{eq:theta}
	\theta_{r,\varepsilon,x_0}(x) = 
	\begin{cases}
		1 \quad & \text{if } x \in B(x_0,r),
		\\
		0 \quad & \text{if } x \not\in B(x_0,r+\varepsilon),
		\\
		0<\theta_{r,\varepsilon,x_0}(x)< 1 \quad& \text{otherwise}.
	\end{cases}
\end{equation*}

And insert these functions into %the inequality
\eqref{neq:semicont}
in place of
$\theta(x)$.
Passing to the limit as
$\varepsilon \rightarrow 0$,
we obtain
\begin{multline*}
	\int\limits_{B(y,r)} |D\varphi_0 (x)|^{\frac{ns}{s+1}} \, dx = 
	\lim\limits_{\varepsilon \rightarrow 0}\int\limits_{\Omega} 
	|D\varphi_0 (x)|^{\frac{ns}{s+1}} \theta_{r,\varepsilon, y}(x) \,dx
	\\ \leq 
	\Biggl( \lim\limits_{\varepsilon \rightarrow 0} \int\limits_{\Omega} 
	M^{s}(x)  \theta_{r,\varepsilon, y}(x) \,dx \Biggr)^{\frac{1}{s+1}}
	\Biggl( \lim\limits_{\varepsilon \rightarrow 0} \int\limits_{\Omega} 
	J(x, \varphi_0) \theta_{r,\varepsilon, y}(x)\,dx \Biggr)^{\frac{s}{s+1}}
	\\ = 
	\Biggl( \int\limits_{B(y,r)} M^{s}(x) \,dx \Biggr)^{\frac{1}{s+1}}
	\Biggl(\int\limits_{B(y,r)} J(x, \varphi_0) \,dx \Biggr)^{\frac{s}{s+1}}.
\end{multline*}

Then,
divide by the measure of the ball
and pass to the limit as
$r \rightarrow 0$:
\begin{multline*}
	|D\varphi_0 (y)|^{\frac{ns}{s+1}} =
	\lim\limits_{r \rightarrow 0} \frac{1}{|B(y,r)|}
	\int\limits_{B(y,r)} |D\varphi_0 (x)|^{\frac{ns}{s+1}} \, dx  
	\\ \leq
	\Biggl( \lim\limits_{r \rightarrow 0} \frac{1}{|B(y,r)|} 
	\int\limits_{B(y,r)} M^{s}(x) \,dx \Biggr)^{\frac{1}{s+1}}
	\\ \times
	\Biggl(\lim\limits_{r \rightarrow 0} \frac{1}{|B(y,r)|}
	\int\limits_{B(y,r)} J(x, \varphi_0) \,dx \Biggr)^{\frac{s}{s+1}}
	\\=
	M^{s}(y)^{\frac{1}{s+1}} J(y, \varphi_0)^{\frac{s}{s+1}}
\end{multline*}
for almost all
$y\in \Omega$.

Thus,
the pointwise inequality 
\begin{equation*}
	|D\varphi_0 (x)|^{\frac{ns}{s+1}} \leq M^{s}(x)^{\frac{1}{s+1}}
	J(x, \varphi_0)^{\frac{s}{s+1}}  
	\text{ for almost all } x \in \Omega
\end{equation*}
holds.
It implies the assertion of the lemma:
\begin{equation*}
	\frac{ |D\varphi_0 (x)|^n}{J(x, \varphi_0)} \leq M(x)  
	\text{ for almost all } x \in \Omega.
\end{equation*}
\end{proof}

\subsubsection{Semicontinuity of the functional.}\label{subsec:semicont}

In order to complete the proof,
it remains to verify that
\begin{equation}\label{l.s.c.}
	\int\limits_\Omega W(x, D\varphi_0)\,dx\leq \varliminf\limits_{k\rightarrow\infty} 
	\int\limits_\Omega W(x,D\varphi_k)\,dx.
\end{equation}

If the right-hand side equals~%
$\infty$
then the inequality is obvious.
If the~right-hand side is finite
then there exists a~subsequence
$\varphi_m$
for which the~sequence
$\biggl\{\int\limits_\Omega W(x, D\varphi_m)\,dx\biggr\}$
converges.

Using the weak convergence~\eqref{weak_lim} and Mazur theorem,
we see that
for each~%
$m$
we can find an~integer
$j(m)\geq m$
and real numbers
$\mu_m^t>0$
for
$m\leq t\leq j(m)$
such that
\begin{equation*}
	\sum\limits_{t=m}\limits^{j(m)}\mu_m^t=1,
\end{equation*}
and in
$L_n(\Omega)\times L_{\frac{n}{n-1}}(\Omega) \times L_r(\Omega)$
we have
\begin{equation*}
	D_m=\sum\limits_{t=m}\limits^{j(m)}\mu_m^t(D\varphi_t,\Adj
	{D\varphi_t}, J(\varphi_t)) \xrightarrow[m \rightarrow \infty]{} 
	(D\varphi_0,\Adj D\varphi_0,J(\cdot,\varphi_0)).
\end{equation*}
Consequently,
there exists a~subsequence
$\{D_l\}$
converging almost everywhere in~%
$\Omega$.

Since
$G$
satisfies Carath\'eodory conditions,
it follows that
$G(x,\cdot)$
is continuous for almost all
$x\in\Omega$
and
\begin{multline*}
	W(x, D\varphi_0)=G(x,D\varphi_0,\Adj{D\varphi_0},J(\cdot,\varphi_0))
	\\ = 
	\lim\limits_{l\rightarrow\infty}G\bigg(x,\sum\limits_{t=l}\limits^{j(l)}\mu_l^t
	(D\varphi_t, \Adj{D\varphi_t}, J(\cdot,\varphi_t))\bigg).
\end{multline*}

Applying Fatou's lemma and the convexity of~% the function
$G$,
we arrive at
\begin{multline*}
	\int\limits_\Omega W(x, D\varphi_0)\,dx \leq \varliminf\limits_{l\rightarrow\infty}
	\int\limits_\Omega G\bigg(x,\sum\limits_{t=l}\limits^{j(l)}\mu_l^t(D\varphi_t,
	\Adj{D\varphi_t}, J(\cdot,\varphi_t))\bigg)\,dx 
	\\ \leq
	\varliminf\limits_{l\rightarrow\infty}\sum\limits_{t=l}\limits^{j(l)}\mu_l^t 
	\int\limits_\Omega  G(x,D\varphi_t,\Adj{D\varphi_t}, J(\cdot,\varphi_t))\,dx 
	= \lim\limits_{l\rightarrow\infty}W(x,D\varphi_l).
\end{multline*}
This justifies %the inequality
\eqref{l.s.c.}.

\section{Examples}\label{sec:examples}

As our first example
consider an~Ogden material
with the stored-energy function
$W_1$
of the form
\begin{equation}\label{example1}
	W_1(F)=a\tr(F^T F)^{\frac{p}{2}} + b \tr\Adj (F^T F)^{\frac{q}{2}} 
	+ c(\det F)^r + d (\det F)^{-s},
\end{equation}
where
$a>0$,
$b>0$,
$c>0$,
$d>0$,
$p> 3$,
$q >3$,
$r>1$,
and
$s>\frac{2q}{q-3}$.
Then
$W_1(F)$
is polyconvex
and the coercivity inequality holds
\cite[Theorem 4.9-2]{Ciarlet1992}:
\begin{equation*}
	W_1(F)\geq \alpha\big(\|F\|^p+\|\Adj F\|^q\big)+ c(\det F)^r + d (\det F)^{-s}.
\end{equation*}
Assume also that
boundary conditions on the displacements are specified,
and furthermore,
$\overline{\varphi}: \overline{\Omega} \to \overline{\Omega'}$
is a~homeomorphism.
We have to solve the~minimization problem 
\begin{equation*}
	I_1(\varphi_b)=\inf\limits_{\varphi\in\A_B} I_1(\varphi),
\end{equation*}
where
$I_1(\varphi) = \int\limits_{\Omega} W_1 (D\varphi (x)) \, dx$
and the class of admissible deformations
$\A_B$
is defined in~\eqref{def:AB}.
The result of John Ball~\cite{Ball1981} ensures that
there exists at least one solution
$\varphi_B\in\A_B$
to this problem,
which in addition is a~homeomorphism.

On the other hand,
for the functions of the form~\eqref{example1}
Theorem~\ref{thm:main} holds.
Indeed,
$W_1(F)$
is polyconvex and satisfies %the inequality
\begin{equation*}
W_1(F)\geq \alpha \|F\|^n + c(\det F)^r - \alpha,
\end{equation*}
where~%
$\alpha$
plays the role of the function
$h(x)$
of~\eqref{neq:coer}.
When we consider the~same boundary conditions 
$\overline{\varphi}: \overline{\Omega} \to \overline{\Omega'}$
and have to solve the minimization problem 
\begin{equation*}
	I_1(\varphi_0)=\inf\limits_{\varphi\in\A} I_1(\varphi)
\end{equation*}
in the class of admissible deformations~%
$\A$
defined by~\eqref{def:A},
Theorem~\ref{thm:main} yields a~solution
$\varphi_0\in\A$
which is a~homeomorphism.

This example shows that
in those problems to which Ball's theorem applies
the hypotheses of Theorem~\ref{thm:main} are fulfilled;
consequently,
we can consider admissible deformations of class~%
$\A$.

Let us discuss another example.
Here the stored-energy function is of the~form
\begin{equation*}%\label{example2}
	W_2(F)=a \, {\tr} (F^T F)^{\frac{3}{2}}+ c (\det F)^{r}.
\end{equation*}
This function is polyconvex and satisfies %the inequality
\begin{equation*}
	W_2(F)\geq \alpha \|F\|^3+ c (\det F)^r,
\end{equation*}
but violates the
 inequality of the form~\eqref{neq:coer_b}.
Moreover,
$W_2(F)$
violates the~asymptotic condition
\begin{equation*}
	W_2(x,F)\longrightarrow\infty
	\text{ as } \det F \longrightarrow 0_+,
\end{equation*}
which plays an~important role in~\cite{Ball1981, BallCurOl1981}
and other articles.

Nevertheless,
for the stored-energy function
$W_2$
there exists a~solution to the minimization problem 
$	
I_2(\varphi_0)=\inf\limits_{\varphi\in\A} I_2(\varphi)
$,
where
$I_2 (\varphi) = \int\limits_{\Omega} W_2 (D\varphi (x)) \, dx$,
which is a~homeomorphism 
provided that
a~homeomorphism~%
$\overline{\varphi}$
is prescribed on the boundary.

\end{document}